\documentclass[11pt, a4paper, UKenglish]{amsart}
\usepackage[a4paper, margin=3.5cm]{geometry}
\usepackage[utf8]{inputenc}
\usepackage[T1]{fontenc}
\usepackage{enumerate}
\usepackage{xcolor}
\usepackage{imakeidx}
\usepackage{amsfonts}
\usepackage{amssymb}
\usepackage{amsmath}
\usepackage{amsthm}
\usepackage{tikz-cd}
\usepackage{float}
\usepackage{graphicx}
\usepackage{tikz}
\usepackage{thmtools}
\usepackage{thm-restate}
\usepackage{mathrsfs}
\usepackage{enumitem}
\usepackage{mathtools}
\usepackage{wrapfig}
\usepackage{verbatim}
\usepackage[all]{xy}
  \SelectTips{cm}{10}     
  \everyxy={<2.5em,0em>:} 

\usepackage[colorlinks=true,hyperindex,breaklinks]{hyperref}
\usepackage[backend=biber, style=alphabetic, sorting=nyt, maxcitenames=10, minbibnames=10, maxbibnames=10]{biblatex}
\usepackage{xurl}
\hypersetup{
           breaklinks=true,   
           colorlinks=true,   
           pdfusetitle=true,  
        }

\theoremstyle{definition}
\newtheorem{theorem}{Theorem}[section]
\newtheorem{definition}[theorem]{Definition}

\newtheorem*{notation*}{Notation}

\newtheorem{remark}[theorem]{Remark}
\newtheorem{corollary}[theorem]{Corollary}
\newtheorem{proposition}[theorem]{Proposition}

\newtheorem{lemma}[theorem]{Lemma}

\newtheorem{claim}[theorem]{Claim}

\newtheorem*{question*}{Question I}
\newtheorem*{question'}{Question II}
\newtheorem*{question''}{Question III}

\newtheorem*{theorem*}{Theorem}

\newcommand{\ZZ}{{\mathbb Z}}
\newcommand{\NN}{{\mathbb N}}
\newcommand{\RR}{{\mathbb R}}
\newcommand{\QQ}{{\mathbb Q}}

\newcommand{\PP}{{\mathbb P}}
\newcommand{\bbS}{{\mathbb S}}

\newcommand{\bbC}{\mathbb{C}}

\newcommand{\bbN}{\mathbb{N}}
\newcommand{\bbP}{\mathbb{P}}
\newcommand{\bbQ}{\mathbb{Q}}
\newcommand{\bbR}{\mathbb{R}}
\newcommand{\bbZ}{\mathbb{Z}}

\newcommand{\bbG}{\mathbb{G}}
\newcommand{\bbT}{\mathbb{T}}

\newcommand{\cX}{{\mathcal X}}

\newcommand{\cO}{{\mathcal O}}

\newcommand{\cL}{{\mathcal L}}

\newcommand{\cA}{{\mathcal A}}

\newcommand{\sL}{\mathcal{L}}

\newcommand{\sO}{\mathcal{O}}

\newcommand{\sR}{\mathcal{R}}

\newcommand{\sX}{\mathcal{X}}
\newcommand{\sY}{\mathcal{Y}}

\DeclareMathOperator{\Spec}{Spec}

\DeclareMathOperator{\Sym}{Sym}

\DeclareMathOperator{\vol}{vol}

\DeclareMathOperator{\qf}{qf}

\DeclareMathOperator{\MI}{MI}
\DeclareMathOperator{\MV}{MV}

\DeclareMathOperator{\GVF}{GVF}

\DeclareMathOperator{\an}{an}
\DeclareMathOperator{\FS}{FS}

\DeclareMathOperator{\trop}{trop}

\DeclareMathOperator{\cdiv}{div} 

\DeclareMathOperator{\AdivQ}{ADiv_{\QQ}}

\DeclareMathOperator{\LdivQ}{LDiv_{\QQ}}

\DeclareMathOperator{\LPic}{LPic}
\DeclareMathOperator{\LPicQ}{LPic_{\QQ}}

\DeclareMathOperator{\Pic}{Pic}
\DeclareMathOperator{\ZDiv}{ZDiv} 
\DeclareMathOperator{\ZDivQ}{ZDiv_{\QQ}} 
\DeclareMathOperator{\height}{ht}

\DeclareMathOperator{\M}{M}

\DeclareMathOperator{\can}{can}

\DeclareSymbolFont{yhlargesymbols}{OMX}{yhex}{m}{n} \DeclareMathAccent{\yhwidehat}{\mathord}{yhlargesymbols}{"62}

\newcommand{\Qbar}{\ov{\mathbb{Q}}}

\newcommand{\Pvariety}{X}

\newcommand{\ac}{\widehat{c}}

\newcommand{\Field}{F}

\newcommand{\adiv}{\widehat{\cdiv}}

\newcommand{\aPicQ}{\widehat{\Pic}_\bbQ}
\newcommand{\intPicQ}{\widehat{\Pic}^{\text{int}}_\bbQ}

\newcommand{\adeg}{\widehat{\deg}}

\newcommand{\GVFQ}{\ov{\QQ}}
\newcommand{\GVFan}[2][]{\ensuremath{#2_{\GVF\ifx&#1& \else ,#1 \fi}}}

\newcommand{\ov}{\overline}

\newcommand{\meet}{\wedge}

\newcommand{\blank}{{-}}

\let\svthefootnote\thefootnote
\newcommand\freefootnote[1]{%
  \let\thefootnote\relax%
  \footnotetext{#1}%
  \let\thefootnote\svthefootnote%
}

\newcommand{\Kbar}{\ov{K}}

\title[Continuity of heights \& complete intersections in toric varieties]{Continuity of heights in families and complete intersections in toric varieties}
\date{}
\author{Pablo Destic}
\address{Institut Camille Jordan, Université Claude Bernard Lyon 1, 43 boulevard du 11 novembre 1918 F-69622 Villeurbanne Cedex, France}
\email{\url{destic@math.univ-lyon1.fr}}
\author{Nuno Hultberg}
\address{Institute for Theoretical Sciences, Westlake University, No. 600 Dunyu Road, Sandun town, Xihu district, Hangzhou, Zhejiang Province, 310030 China}
\email{\url{nuno.hultberg@outlook.com}}
\author{Michał Szachniewicz}
\address{Mathematical Institute, University of Oxford, Oxford, OX2 6GG, UK}
\email{\url{michal.szachniewicz@maths.ox.ac.uk}}


\begin{document}

\begin{abstract}
    We study the variation of heights of cycles in flat families over number fields or, more generally, globally valued fields. To a finite type scheme $S$ over a GVF $K$ we associate a locally compact Hausdorff space $S_{\GVF}$ which we refer to as the GVF analytification of $S$. For a flat projective family $\cX \subset \PP_S^n \to S$, we prove that $(s \in S_{\GVF}) \mapsto \height(\cX_s)$ is continuous.
    
    As an application, we prove Roberto Gualdi's conjecture on limit heights of complete intersections in toric varieties.
\end{abstract}

\maketitle

\tableofcontents

\freefootnote{An earlier version of this article was included in the second author's PhD thesis. The main difference is the additional treatment of the trivial valuation.}

\section{Introduction}
In classical algebraic geometry Bezout's theorem states that generically the intersection of $n$ hypersurfaces of degrees $d_1,\dots,d_r$ in $\bbP^n$ is of degree $d_1\dots d_r$. If the hypersurfaces are defined over a number field, one may ask whether it is possible to compute the Weil height of the intersection in terms of their arithmetic complexity. This is important, because estimates on Weil heights can lead to finiteness theorems about rational points. A striking example of this philosophy appears in (Proposition 2.17 of) \cite{faltings1991diophantine}. Another example is the arithmetic Bezout's theorem obtained in \cite{bost1994heights}, however it only gives an upper bound for the height of an intersection. Since Weil heights are also connected to special values of L-functions and periods (see e.g. Section 4 of \cite{pazuki2024northcott}, or \cite{maillot2000geometrie, CASSAIGNE2000226}) it is desirable to have formulas for their exact values. The starting point of our considerations is \cite{gualdi_hypersurfaces_in_toric_varieties}, where such a formula was given, for the height of a single hypersurface ($r=1$) in a toric variety, with respect to a semipositively metrized toric divisor.

However, Roberto Gualdi observed in his thesis \cite{gualdi:tel-01931089} that it is not possible to extend a result of the same nature to $r\geq2$. In fact, he gives examples of polynomials with the same associated arithmetic data, defining cycles of differing heights. As a remedy to this issue, he suggested to consider average heights of cycles with prescribed arithmetic data and formulated the following conjecture, which we prove in this article.

\begin{theorem}[Theorem~\ref{thm:gualdi_conjecture}]
    Let $f_1, \dots, f_m$ be Laurent polynomials in $n$ variables with coefficients in a number field $K$ and let $T$ be a proper toric variety with torus $\bbT=\bbG^n \subset T$. Denote by $V_i$ the hypersurface defined by $f_i$ and by $\rho_i$ its Ronkin function. Let $(\zeta_{1,j},\dots,\zeta_{m,j})_j$ be a generic sequence of small points in $\bbT^m$ for the Weil height and let $\ov{D}_0,\dots,\ov{D}_{n-m}$ be semipositive toric adelic divisors on $T$ with associated local roof functions $\theta_{0,v}, \dots, \theta_{n-m,v}$. Then,
    \[
        \lim_{j\to\infty} \adeg(\ov{D}_0, \dots,\ov{D}_{n-m}\mid \zeta_{1,j}V_1\cap\dots\cap\zeta_{m,j}V_m)
    \]
    \[
        =\sum_{v\in \M_K} n_v \MI_M(\theta_{0,v}, \dots, \theta_{n-m,v}, \rho_1^\vee, \dots, \rho_m^\vee).
    \]
\end{theorem}

The original conjecture can be found in \cite[Conjecture 6.4.4]{gualdi:tel-01931089}. Let us briefly describe its contents. The arithmetic degree of a suitably generic complete intersection subvariety in a toric variety can be computed as an arithmetic intersection number of adelic toric divisors. The formulation of the conjecture is in terms of a convex geometry identity for this intersection number involving mixed integrals. Mixed integrals and further convex geometry tools are discussed in Section \ref{sec:convex_geometry}.

Roberto Gualdi and Martin Sombra have together proved partial results in this direction. In \cite{gualdi_sombra_limit_heights} they prove the above result in the case $T=\bbP^2$, $m=2$ and $f_1(x_1,x_2) = f_2(x_1,x_2) = x_1+x_2+1$. Moreover, for this example, they compute that both sides of the identity in question are equal to the intriguing value $\frac{2\zeta(3)}{3\zeta(2)}$. They have also solved the $m=2$ case of the conjecture in \cite{gualdi_sombra_codimension2}. Their methods fundamentally differ from ours. In particular, their approach relies on local logarithmic equidistribution as in \cite{dimitrov_habegger_galois_orbits_of_torsion} while we modify the problem in such a way that we can apply Yuan's equidistribution theorem from \cite{Yuan_big_line_arithmetic_bundles_Siu_inequality}.

Said modification can be conveniently phrased in the framework of globally valued fields (abbreviated GVF). Globally valued fields were introduced by Ben Yaacov and Hrushovski to serve as a theory of fields with multiple valuations satisfying the product formula. As such, globally valued fields are closely related to proper adelic curves as defined and developed in \cite{Adelic_curves_1} by Chen and Moriwaki, see \cite[Corollary 1.3]{basics_of_gvfs} for the precise relationship. Globally valued fields, however, form a theory in unbounded continuous logic from \cite{Ben_Yaacov_unbdd_cont_FOL}, and are therefore amenable to methods from model theory.

The original motivation of this article was to prove the definability of intersection products of divisors over globally valued fields, i.e., that arithmetic intersection numbers parametrized over a base form a quantifier-free definable formula. For the benefit of Arakelov geometers reading this article, we will phrase this as a continuity result.\footnote{We thank R\'emi Reboulet for pointing out the similarity to the existence of a Deligne pairing. In fact, in the number field case our result is likely to be implied by the Deligne pairing in \cite[Theorem 4.1.3]{yuan_zhang_adeliclinebundlesquasiprojective}. Our approach has the advantage of its simplicity.} First, let us state a version of the `continuity of heights' over $\Qbar$ that does not require any additional definitions.

\begin{theorem}
    Let $\pi:\sX \to S$ be a surjective morphism of projective varieties over $\Qbar$, generically of relative dimension $d$. Let $(s_i)_i$ be a generic sequence of points in $S(\Qbar)$ such that for every adelic line bundle $\ov{M}$ on $S$ the value $h_{\ov{M}}(s_i)$ converges. Let $\ov{L}_0,\ldots, \ov{L}_d$ be integrable line bundles on $\sX$.
    Then, the arithmetic intersection number on fibers
    \[
    \adeg(\ov{L}_0 \ldots \ov{L}_d|\cX_{s_i})
    \]
    converges. Its limit can be described as a certain arithmetic intersection number over a GVF.
    
    If $h_{\ov{M}}(s_i)$ converges to $\ov{L}^{\dim S}\ov{M}$ for some arithmetically nef line bundle $\ov{L}$ and every adelically metrized line bundle $\ov{M}$, then $\adeg(\ov{L}_0 \ldots \ov{L}_d|\cX_{s_i})$ converges to the intersection number $(\pi^*\ov{L})^{\dim S}\ov{L}_0\ldots \ov{L}_d$.
\end{theorem}

The first part corresponds to Theorem~\ref{thm:continuity_of_intersection_number} (which we state in detail below) and the latter part to Proposition~\ref{proposition_calculating_generic_intersection}, essentially due to Chen and Moriwaki \cite[Proposition 4.5.1]{Adelic_curves_2}. Equidistribution theorems in Arakelov geometry provide us with various examples of sequences to which the above theorem can be applied. We note in particular Yuan's equidistribution in the form of \cite[Lemma 8.2]{Chambert_Loir_survey} gives such examples.

Before stating the `continuity of heights' Theorem \ref{thm:continuity_of_intersection_number} let us recall its algebro-geometric counterpart which is the following constructibility/constantness result.

\begin{theorem}[Corollary III 9.10 \cite{Hartshorne_AG}]
    Let $\sX \to S$ be a flat projective family of varieties over a field $K$ of relative dimension $d$. Let $\sL$ be a line bundle on $\sX$. Then, the degree of the fibres $\deg_{\sL}(\sX_s)$ is locally constant.
\end{theorem}

Note that this is equivalent to the map $s\mapsto \deg_{\sL}(\sX_s)$ being a continuous function on $S$ with the Zariski topology. The same result is implied for the intersection of $d$ line bundles. In our result the Zariski topology is replaced by a GVF analytic topology, and the degree is replaced by height. Let us be more precise.

We associate to a finite type scheme $S$ over a GVF $K$, a locally compact Hausdorff space $S_{\GVF}$, called its GVF analytification. This is done similarly to the Berkovich analytification, however with valuations replaced by global heights (see Definition~\ref{definition:GVF} and Section \ref{sec:gvf_analytification}). For $K = \Qbar$, a generic sequence $s_i$ in $S(\Qbar)\subset S_{\GVF}$ converges to a point in $S_{\GVF}$ if and only if the value $h_{\ov{M}}(s_i)$ converges for every adelically metrized line bundle $\ov{M}$ on $S$. In Definition~\ref{definition:global_line_bundles}, for a flat family $\cX \to S$ over a GVF $K$, we introduce the group of global line bundles on $\cX$ over $S$. Following \cite{Zhang_thesis_inequality}, we define semipositivity and integrability in this context. For example, the group of globally integrable line bundles on $\cX$ over $S$ is denoted by $\intPicQ(\cX/S)$. In Proposition~\ref{proposition:intersection_product_over_a_GVF} we define an intersection paring on tuples of globally integrable line bundles and we prove that is satisfies the following.
\begin{theorem}[Theorem~\ref{thm:continuity_of_intersection_number}]~\label{thm:intro_continuity_of_intersection_number}
    Let $\cX \to S$  be a flat projective morphism of finite type schemes over a GVF $K$ of relative dimension $d$. Let $\ov{\sL}_0,\dots,\ov{\sL}_d\in \intPicQ(\cX/S)$. Then, the map
    \begin{align*}
        S_{\GVF} &\to \bbR\\
        s&\mapsto\adeg(\ov{\sL}_0,\dots,\ov{\sL}_d|\sX_s)
    \end{align*}
    is continuous.
\end{theorem}
Equivalently, this means that the intersection product $\adeg(\ov{\sL}_0,\dots,\ov{\sL}_d|\sX_s)$ (for $s \in S(F)$ with $K \subset F$ being a GVF extension) can be defined by a quantifier-free formula in the GVF language, with parameters from the base-field $K$. We believe that this fact may be important in axiomatizing the model companion of globally valued fields (if it exists).

The structure of this text is the following. In Section~\ref{sec:preliminaries} we introduce globally valued fields, GVF analytifications, and give some examples. We also present necessary notions from Arakelov geometry of toric varieties. In Section~\ref{sec:int_product} we use the theory of adelic curves to define the intersection product over arbitrary GVF and prove Theorem~\ref{thm:continuity_of_intersection_number}. We also prove Theorem~\ref{theorem:Zhang_divisors_are_induced_by_global_divisors} which relates global integrable line bundles (in our new sense) to integrable adelic line bundles over a number field (in the sense of Zhang). In Section~\ref{sec:Average_intersections} we use the previous results and perform calculations which allow us to conclude with Theorem~\ref{thm:gualdi_conjecture}. In Appendix~\ref{appendix:A} we prove an estimate on a variant of the Mahler measure of a polynomial, needed in the proof of Theorem~\ref{thm:continuity_of_intersection_number}. We believe this is known to experts, however we could not find a suitable reference. The appendix is self-contained and elementary.
\subsection{Acknowledgments}
We would like to thank our advisors Itaï Ben Yaacov, Ehud Hrushovski and Fabien Pazuki for various discussions and comments regarding this work. We further thank Klaus Künnemann for useful remarks on the exposition. Special thanks go to Roberto Gualdi and Martin Sombra, for sharing with us their draft \cite{gualdi_sombra_codimension2} and for their helpful remarks. Moreover, we would like to thank the organizers and the participants of the 2024 Students’ Conference on Non-Archimedean, Tropical and Arakelov Geometry, for their interest in this project. We are grateful to R\'emi Reboulet for pointing out the connection to Deligne pairings. The third author would like to thank Douglas Molin for discussions about limit heights. The second author would like to thank Qingyuan Bai for helpful suggestions on mixed volumes.

\section{Preliminaries}~\label{sec:preliminaries}
This section will serve both to recall definitions and theorems, as well as to state and prove technical results related to the definitions just recalled.
\subsection{Globally valued fields}

Globally valued fields are a theory in (unbounded) continuous logic designed to model fields with multiple valuations and a product formula. They are closely related to adelic curves.

There are multiple ways to describe globally valued fields, see \cite{basics_of_gvfs} for an introduction. The simplest definition is the following.
\begin{definition}~\label{definition:GVF}
    A \emph{globally valued field} (abbreviated GVF) is a field $F$ together with a height function $h:\mathbb{A}(F)\rightarrow \mathbb{R}\cup\{-\infty\}$, where $\mathbb{A}(F)$ denotes the disjoint union of $\mathbb{A}^n(F)$ for all $n\in\mathbb{N}$, satisfying the following axioms, for some \emph{Archimedean error} $e \geq 0$.
\[
    \begin{array}{lll}
        \textnormal{Height of zero:} & \forall x\in F^n, & h(x) = -\infty \Leftrightarrow x=0 \\
        \textnormal{Height of one:} & & h(1,1) = 0 \\
        \textnormal{Invariance:} & \forall x\in F^n,\, \forall \sigma\in \Sym_n, & h(\sigma x) = h(x) \\
        \textnormal{Additivity:} & \forall x\in F^n,\, \forall y\in F^m, & h(x\otimes y) = h(x) + h(y) \\
        \textnormal{Monotonicity:} & \forall x\in F^n,\, \forall y\in F^m, & h(x) \leq h(x,y) \\
        \textnormal{Triangle inequality:} & \forall x,y\in F^n, & h(x+y) \leq h(x,y) + e \\
        \textnormal{Product formula:} & \forall x\in F^{\times}, & h(x) = 0
    \end{array}
\]
    Here $\otimes$ denotes the Segre product, i.e., $(x_1, \dots, x_n) \otimes (y_1, \dots, y_m) = (x_i \cdot y_j : 1 \leq i \leq n, 1 \leq j \leq m)$. Note that such height factors through $h:\mathbb{P}^n(F)\rightarrow \mathbb{R}_{\geq 0}$ for each $n$. We write $\height(x) := h[x:1]$ for $x \in \Field$.
\end{definition}

If $\Field$ is countable, these can also be seen as equivalence classes of proper adelic curve structures on $\Field$ (originally defined in \cite{Adelic_curves_1}).

\begin{definition}
    A \textit{proper adelic curve} is a field $\Field$ together with a measure space $(\Omega, \mathcal{A}, \nu)$ and with a map $(\omega \mapsto |\cdot|_{\omega}):\Omega \to M_{\Field}$ to the space of absolute values on $\Field$, such that for all $a \in \Field^{\times}$ the function 
    \[ \omega \mapsto \omega(a) := -\log|a|_{\omega} \]
    is in $L^1(\nu)$ with integral zero.
\end{definition}

If $\Field$ is equipped with a proper adelic curve structure, then one can define a GVF structure on it, by putting
\[ h(x_1, \dots, x_n) := \int_{\Omega} -\min_i(\omega(x_i)) d\nu(\omega). \]
On the other hand, if $\Field$ is countable, any GVF structure on $\Field$ is represented by some proper adelic curve structure in this way (\cite[Corollary 1.3]{basics_of_gvfs}). 

There is yet another equivalent way to describe a GVF structure on a field which turns out convenient for our purposes. We give here this definition in the case where $F$ is countable. The general definition can be found in \cite{basics_of_gvfs}.

\begin{definition}
    A \emph{lattice group} is a partially ordered abelian group $(G,+,\le)$ such that every pair of elements $x,y\in G$ has a greatest lower bound, denoted $x\meet y$. Since the order may be recovered from the binary operator $\meet$, we may also call the triple $(G,+,\meet)$ a lattice group.

    We equip the space $(M_F)^{\mathbb{R}}$ with the structure of a lattice group such that $f : M_F \rightarrow \mathbb{R}$ is positive if and only if $f(|\cdot|)\geq 0$ for each absolute value $|\cdot|$. We call the space of \emph{lattice divisors} over $F$, denoted $\LdivQ(F)$, the divisible lattice subgroup of $(M_F)^{\mathbb{R}}$ generated by elements of the form $\adiv(x) : |\cdot| \mapsto -\log|x|$ for $x\in F^\times$. These generators are called \emph{principal} lattice divisors.
\end{definition}

Then, GVF structures on $\Field$ correspond to so called \textit{GVF functionals}, which are linear functionals
\[ l:\LdivQ(\Field) \to \RR \]
that are non-negative on the positive cone, and are zero on principal lattice divisors.

Assume that $\Field$ is a finitely generated extension of $\QQ$. In \cite{szachniewicz2023} and \cite{basics_of_gvfs} an Arakelov theoretic interpretation of the lattice $\LdivQ(\Field)$ was given. More precisely, $\LdivQ(\Field)$ embeds into
\[ \AdivQ(\Field) = \varinjlim \AdivQ(\cX), \]
where $\AdivQ(\cX)$ is the group of arithmetic $\QQ$-divisors of $C^0$-type on $\cX$, and the union is taken over the system of all arithmetic varieties $\cX$ (i.e., normal, integral, flat and projective over $\Spec(\ZZ)$) with an isomorphism $\kappa(\cX) \simeq \Field$. In \cite{szachniewicz2023} a group lattice structure on $\AdivQ(\Field)$ was defined, so that the embedding is a morphism of group lattices. 

Let $X = (X,\pi_X)$ be a projective scheme over $\QQ$ with a fixed isomorphism $\pi_X : \kappa(X) \rightarrow \Field$. To avoid confusion, we use the name \textit{Zhang divisors/Zhang line bundles} for Cartier divisors on $X$ equipped with adelic Green functions/line bundles with adelic metrics, in the sense of \cite[Remark (4.8)]{Chambert_Loir_survey} (originally defined by Zhang \cite{zhang1995small}). We denote the group of Zhang divisors on $X$ by $\ZDiv(X)$ and by $\ZDivQ(X)$ its tensor with $\QQ$. Using the same methods as in \cite{szachniewicz2023}, one also defines a lattice structure on
\[ \ZDivQ(\Field) = \varinjlim \ZDivQ(X), \]
where the direct limit is taken over the system of all $X$ as above (with maps respecting the isomorphisms $\pi_X$). Then the group lattice $\LdivQ(\Field)$ has a natural embedding into $\ZDivQ(\Field)$. Moreover, every GVF functional extends uniquely to a positive functional on $\ZDivQ(\Field)$ so we get the following.
\begin{corollary}
    On a finitely generated field $\Field$ over $\QQ$ there is a natural bijection
    \[ \bigl\{ \textnormal{GVF functionals $\ZDiv_{\QQ}(\Field) \to \RR$}  \bigr\} \longleftrightarrow \bigl\{ \textnormal{GVF structures on $\Field$}  \bigr\}, \]
    where by GVF functionals we mean the linear ones that are non-negative on the effective cone, and are zero on principal Zhang divisors.
\end{corollary}
For $\ov{D} \in \ZDivQ(X)$ and $x \in X(\GVFQ)$ we write $h_{\ov{D}}(x)$ for the height of $x$ with respect to $\ov{D}$. A GVF functional $l$ determines a GVF structure on a field by the formula
\[ h(x_1, \dots, x_n) := l\bigl(-\bigwedge_{i=1}^n \adiv(x_i)\bigr). \]
Globally valued fields form a category where maps are embeddings of fields respecting height functions. Equivalently, an extension of fields $K \subset \Field$ induces an embedding $\LdivQ(K) \subset \LdivQ(\Field)$ and GVF structures on $\Field$ extending a given GVF structure on $K$ are precisely extensions of GVF functionals. Let us point out the following fact.

\begin{lemma}~\cite[Lemma 10.2]{basics_of_gvfs}~\label{lemma:uniqueness_of_GVF_on_Q}
    For any $e \geq 0$ there is a unique GVF structure on $\QQ$ (and on any subfield of $\GVFQ$) satisfying $\height(2) = e \cdot \log 2$.
\end{lemma}

We refer to the GVF structure with $e=1$ as the standard one.\\

\subsection{GVF analytification}~\label{sec:gvf_analytification}
Now, we describe a construction which recovers the space of quantifier-free types in the theory of globally valued fields, as in \cite{basics_of_gvfs}.

\begin{definition}
    Let $K$ be a GVF and let $X$ be a finite type scheme over $K$. We define the \emph{GVF analytification} of $X$ over $K$, denoted $\GVFan[K]{X}$ (or $\GVFan{X}$ if the base GVF is implied), to be the set of couples $x = (\pi(x),h_x)$ where $\pi(x)$ is a point of $X$, and $h_x$ is a height on $\kappa(\pi(x))$ extending the height on $K$. If $x\in \GVFan[K]{X}$ and $U\subset X$ is an open containing $\pi(x)$, we will also denote by $h_x$ the map
\[
    h_x : \begin{array}{ccc}
            \mathbb{A}(\mathcal{O}_X(U)) & \rightarrow & \mathbb{R}\cup\{-\infty\} \\
            (f_1,\ldots,f_n) & \mapsto & h_x(f_1(\pi(x)),\ldots,f_n(\pi(x)))
          \end{array}.
\]

We equip $\GVFan[K]{X}$ with the weakest topology such that
\begin{enumerate}
    \item The map $\pi : \GVFan[K]{X} \rightarrow X$ is continuous onto $X$ with the constructible topology.
    \item For every Zariski open $U\subset X$, and every tuple $(f_1,\ldots,f_n)\in \mathcal{O}_X(U)^n$, the map $x \mapsto h_x(f_1,\ldots,f_n)$ is continuous.
\end{enumerate}
\end{definition}

\begin{remark}~\label{remark:GVF_analytification_as_qf_GVF_types}
    If $X\subset \mathbb{A}^n_K$ is affine, it follows that $\GVFan[K]{X}$ can be naturally identified with the space $S_X^{\qf}$ of quantifier-free types on $X$ (see \cite[Construction 11.15]{basics_of_gvfs}). Since the space of quantifier-free types is locally compact, so is $\GVFan[K]{X}$.
    
    It is possible to define a Zariski topology on the GVF analytification by requiring the map to the Zariski topology to be continuous and the height to be merely lower semicontinuous. Then, our methods prove the lower semicontinuity of heights for semipositive line bundles.
\end{remark}

\begin{remark}~\label{remark:embedding_F_points_GVF_analytification}
 Let $K \subset F$ be a GVF extension. Then, there is a canonical \emph{analytification map} $X(F) \rightarrow \GVFan[K]{X}$, defined by taking $x\in X(F)$ to the point $(x,h_x)$, where $h_x$ is the restriction of the height on $F$ to $\kappa(x)$. The image of $x$ by this map will be denoted $x^{\an}$.
\end{remark}

\subsection{Polarisations}

Here we present how arithmetic intersection theory can induce GVF structures and how to interpret Yuan's equidistribution result as certain kind of uniqueness of a GVF structure.

\begin{definition}
    Let $\Field$ be a finitely generated characteristic zero field equipped with a GVF structure. We say that this GVF structure comes from a \textit{polarisation} $(X, \ov{H}_1, \dots, \ov{H}_d)$, if $X$ is a normal projective variety of dimension $d$ over $\QQ$ with function field $\Field$ and $\ov{H}_i \in \ZDivQ(X)$ are arithmetically nef Zhang divisors, such that the GVF functional
    \[ l:\ZDivQ(\Field) \to \RR \]
    is given by
    \[ l(\ov{D}) = \ov{H}_1 \cdot \ldots \cdot \ov{H}_d \cdot \ov{D}. \]
\end{definition}

The name ``polarisation'' comes from \cite[Section 3.1]{moriwaki_finitely_generated_fields}, however here we also allow not necessarily model Zhang divisors. We also use the term polarisation if instead of $\ov{H}_i$'s we are given the corresponding Zhang line bundles $\cO(\ov{H}_i)$. If the $\ov{H}_i$ occurs with multiplicity $k_i$ we denote by $(X, \ov{H}_1^{k_1},\dots,\ov{H}_r^{k_r})$ for the corresponding polarisation.

\begin{remark}
    A polarisation $(X, \ov{H}_1, \dots, \ov{H}_d)$ induces a GVF structure structure on $\Field$ that extends the standard GVF structure on $\QQ$ (i.e. satisfies $\height(2) = \log 2$) if and only if the geometric intersection number satisfies $H_1 \cdot \ldots \cdot H_d = 1$.
\end{remark}

Yuan's equidistribution \cite{Yuan_big_line_arithmetic_bundles_Siu_inequality} gives examples of polarised GVF structures. Let us recall a version of it from \cite[Lemma (8.2)]{Chambert_Loir_survey}.

\begin{theorem}~\label{theorem_Yuan_equidistribution}
    Let $X$ be a projective variety of dimension $d$ over $\QQ$. Fix a semipositive Zhang divisor $\ov{D}$ with $D$ ample and $\ov{D}^{d+1}=0$. For any generic sequence $x_n \in X(\GVFQ)$ with $h_{\ov{D}}(x_n) \to 0$, and any $\ov{M} \in \ZDivQ(X)$ we have
    \[ \lim_n h_{\ov{M}}(x_n) = \frac{\ov{D}^d \cdot \ov{M}}{\deg(D)}. \]
\end{theorem}

For the next two corollaries, fix the context of the above theorem. Also, denote by $\Field$ the function field of $X$.

\begin{corollary}
    There is a unique GVF functional $l$ on $\Field$ extending the standard one on $\QQ$ and satisfying $l(\ov{D}) = 0$.
\end{corollary}
\begin{proof}
    Fix a GVF functional $l$ on $\Field$ with the above properties and a Zhang divisor $\ov{M} \in \ZDivQ(\Field)$. Since the quantities from the assumptions of Theorem~\ref{theorem_Yuan_equidistribution} are birational invariant, without loss of generality $\ov{M} \in \ZDivQ(X)$. By existential closedness of $\GVFQ$ from \cite[Theorem A]{szachniewicz2023} there is a generic sequence of elements $x_n \in X(\GVFQ)$ such that $h_{\ov{D}}(x_n) \to l(\ov{D}) = 0$ and $h_{\ov{M}}(x_n) \to l(\ov{M})$. By Theorem~\ref{theorem_Yuan_equidistribution} we get that 
    \[ l(\ov{M}) = \frac{\ov{D}^d \cdot \ov{M}}{\deg(D)}. \]
    On the other hand, $l$ given by this formula on $X$ and its blowups, is a GVF functional, which finishes the proof.
\end{proof}

\begin{corollary}~\label{corollary:Yuan_equidistribution_as_GVF_convergence}
    Fix a generic sequence $(x_n)_{n \in \NN}$ satisfying the assumptions from Theorem~\ref{theorem_Yuan_equidistribution}. The sequence $x_n^{\an} \in X_{\GVF}$ converges to the point $x_{\infty}^{\an} \in X_{\GVF}$ defined by the generic point $x_{\infty} \in X(\Field)$ together with the GVF structure $l$ on $\Field$. Moreover, $l$ is induced by the polarisation $(X, \ov{D}, \dots, \ov{D})$.
\end{corollary}

Yuan's equidistribution theorem can be applied to the case when $X =\PP_{K}^d$ over a number field $K$ and $\cO(1)$ endowed with the Weil metrics(denoted by $\ov{\cO(1)}$).

\begin{corollary}
    There is a unique GVF structure on $\GVFQ(x_1, \dots, x_d)$ extending the standard one on $\GVFQ$ and satisfying $\height(x_1) = \dots = \height(x_d) = 0$.
\end{corollary}

\begin{corollary}~\label{corollary:Yuan_equidistribution_as_GVF_convergence_for_projective_space}
    Let $\pi_i:\prod_{i=1}^e \PP_K^d \to \PP_K^d$ be the $i$-th projection and let $\ov{L} = \sum_i \pi_i^* \ov{\cO(1)}$.
    Let $x_n \in \prod_{i=1}^e \PP_K^d(\GVFQ)$ be a generic sequence of small points, i.e., satisfying $h_{\ov{L}}(x_n) \to 0$. Let $\Field$ be the function field of $\prod_{i=1}^e \PP_K^d$ with the GVF structure coming from the polarisation $(\prod_{i=1}^e \PP_K^d, \pi_1^* \ov{\cO(1)}^d, \dots, \pi_e^* \ov{\cO(1)}^d)$. Then
    \[ \lim_n x_n^{\an} = x_{\infty}^{\an}, \]
    where $x_{\infty}^{\an} \in \PP_{\GVF}^d$ is defined by the generic point $x_{\infty} \in \PP_K^d(\Field)$. Moreover, naturally identifying $\Field$ with $K(x_{ij}:i\leq d, j\leq e)$, the GVF structure on $\Field$ is the unique one satisfying $\height(x_{ij})=0$ and $\height(2)=\log 2$.
\end{corollary}

\subsection{Convex geometry and toric varieties}~\label{sec:convex_geometry}
This section contains basic notions on Ronkin metrics on toric varieties and some calculations in convex geometry. For more details on Ronkin metrics we suggest \cite{gualdi_hypersurfaces_in_toric_varieties}. For an in depth treatment of the Arakelov geometry of toric varieties, see \cite{burgos_philippon_sombra_toric_varieties}.

Let $\bbT\cong\bbG_m^n$ be a split torus over a number field $K$. Denote by $N$ its co-character lattice and by $M$ its character lattice. Let $X_\Sigma$ denote the proper toric variety associated to a complete fan $\Sigma$ in $N_\bbR$. Let $\bbT \cong X_0 \subset X_\Sigma$ denote the open $\bbT$-orbit in $X_\Sigma$.

\begin{definition}
    A \emph{toric Cartier divisor} on a toric variety $X$ is defined to be a Cartier divisor which is invariant under the action of the torus $\mu:\bbT\times X \to X$. This means that a Cartier divisor $D$ is toric if $\mu^*D = \pi^*D$, where $\pi$ denotes the projection map.
\end{definition}

These are in bijection with virtual polytopes. This bijection is explained and proven in \cite[Section 3.3]{burgos_philippon_sombra_toric_varieties}.

\begin{definition}
    A \emph{virtual support function} or \emph{virtual polytope} with respect to a fan $\Sigma$ on $N_{\bbR}$ is a function $N_{\bbR} \to \bbR$ that is linear and integral on each cone in $\Sigma$.
\end{definition}

Let $v$ be a place of $K$. A $v$-adic Green's function $g_v$ for a toric Cartier divisor $D$ is called toric if its restriction to $(X_0)^{\an}$ factors through the tropicalization map $(X_0)^{\an} \to N_{\bbR}$, defined in \cite[Section 4.1]{burgos_philippon_sombra_toric_varieties}.

\begin{theorem}~\cite[Theorem 4.8.1.(1)]{burgos_philippon_sombra_toric_varieties}~\label{metrics_on_toric_divisors}
    Let $D$ be the toric divisor associated to the virtual support function $\Psi$. Then, the space of $v$-adic Green's functions for $D$ is in bijection with continuous functions $\psi:N_\bbR \to \bbR$ such that $\psi - \Psi$ is bounded. Concave functions correspond to semipositive metrics under this bijection.
\end{theorem}

If $\ov{D}_v$ is a toric divisor with a semipositive $v$-adic Green's function, Legendre-Fenchel duality associates to $\psi$ a concave function on the polytope $D_\Psi \subset M_\bbR$ associated to $\Psi$. It is called the \emph{roof function} of $\ov{D}_v$ and denoted by $\theta_{\ov{D}_v}$. For details, we refer to \cite[Theorem 4.8.1]{burgos_philippon_sombra_toric_varieties}.

A collection of functions $(\psi_v)_{v\in \M_K}$ defines a Zhang metric on the toric divisor corresponding to $\Psi$ if for all $v$ the difference $|\psi_v -\Psi|$ is bounded and for almost all $v$ we have $\psi_v = \Psi$. We call $\ov{D} \in \ZDiv(X_{\Sigma})$ a \textit{toric Zhang divisor}, if $D$ is a toric Cartier divisor, and the metrics on $\ov{D}$ come from a collection of functions $(\psi_v)_{v\in \M_K}$ satisfying the above condition.

The vector space $M_\bbR$ carries a Haar measure normalized in such a way that $M$ has covolume $1$. We associate to a compact convex set $\Delta$ its volume $\vol(\Delta)$ with respect to the Haar measure. Recall that the Minkowski sum of subsets $S_1,S_2$ of a vector space $V$ is defined by
\[
    S_1 + S_2 = \{ s_1 + s_2 \mid s_1 \in S_1, s_2\in S_2\}.
\]

The volume is a homogeneous polynomial on the space of compact convex sets with Minkowski addition. It can therefore be polarized, cf.\ \cite[Definition 2.7.14]{burgos_philippon_sombra_toric_varieties}.
\begin{definition}
    The mixed volume is a multilinear form on compact convex sets defined by
    \[
        \MV(\Delta_1,\dots,\Delta_n) = \sum^n_{j=1} (-1)^{n-j} \sum_{1\leq i_1 < \dots< i_j\leq n} \vol(\Delta_{i_1}+\dots+\Delta_{i_j}).
    \]
    It satisfies $\MV(\Delta,\dots,\Delta) = n!\vol(\Delta)$.
\end{definition}

Given a concave function $\theta$ on a compact convex set $\Delta$, we associate to it its integral $\int_\Delta \theta(m) dm$. Given concave functions $\theta_1$ on $\Delta_1$ and $\theta_2$ on $\Delta_2$, we define their sup-convolution by
\[
    \theta_1 \boxplus \theta_2 (m) = \sup_{m_1+m_2 = m} \theta(m_1) + \theta(m_2).
\]
It defines a concave function on $\Delta_1 + \Delta_2$. It is defined precisely such that the hypograph of $\theta_1 \boxplus \theta_2$ is the Minkowski sum of the hypographs of $\theta_1$ and $\theta_2$. One has a similar polarisation as in the case of mixed volumes, cf.\ \cite[Definition 2.7.16]{burgos_philippon_sombra_toric_varieties}.

\begin{definition}
    The mixed integral is a multilinear form on concave functions $\theta_0,\dots,\theta_n$ on compact convex sets $\Delta_0,\dots,\Delta_n$ defined by
    \[
        \MI(\theta_0,\dots,\theta_n) = \sum^n_{j=0} (-1)^{n-j} \sum_{0\leq i_0 < \dots< i_j\leq n} \int_{\Delta_{i_1}+\dots+\Delta_{i_j}} \theta_{i_1}\boxplus\dots\boxplus\theta_{i_j}(m) dm.
    \]
    It satisfies $\MI(\theta,\dots,\theta) = (n+1)!\int_\Delta \theta(m) dm$.
\end{definition}

These notions are helpful to express arithmetic intersection numbers combinatorially.

\begin{theorem}~\cite[Theorem 5.2.5]{burgos_philippon_sombra_toric_varieties}~\label{thm:convex_formula_for_toric_intersection}
    Let $\ov{D}_0,\dots,\ov{D}_n$ be semipositive toric Zhang divisors on $X_\Sigma$. Then,
    \[
    \adeg(\ov{D}_0,\dots,\ov{D}_n|X_\Sigma) = \sum_{v\in \M_K} n_v \MI_M(\theta_{0,v}, \dots, \theta_{n,v}),
    \]
    where $\theta_{i,v}$ is the roof function of $\ov{D}_{i,v}$, for every $i=0,\dots,n$ and $v \in \M_K$.
\end{theorem}

For a nonzero Laurent polynomial $f \in K[M]$, we denote by $V(f)$ its vanishing locus on $X_\Sigma$. Let $NP(f)$ be the Newton polytope of $f$, i.e. the convex hull of the position of its non-zero coefficients. Then, $NP(f)$ defines a Cartier divisors on a suitable toric modification of $X_\Sigma$ by \cite[Section 3.4]{burgos_philippon_sombra_toric_varieties} such that $f$ gives rise to a regular section by \cite[Section 3.4]{burgos_philippon_sombra_toric_varieties}. It vanishes precisely on $V(f)$. This amounts to an easy check that the section restricts to a nontrivial section on codimension 1 toric subvarieties, see \cite[Theorem 4.3]{gualdi_hypersurfaces_in_toric_varieties}. We now define the Ronkin metric on the divisor $D_{NP(f)}$ corresponding to $NP(f)$. For this we use that the fibers $\trop^{-1}(u)$ of the tropicalization map $\trop:\bbT_v^{\an} \to N_\bbR$ carry a natural choice of probability measure $\sigma_u$. In the non-Archimedean case, it is concentrated on the Gauss point over $u$. In the archimedean case, it is induced by the Haar measure on $(S^1)^n$.

\begin{definition}~\cite[Definition 2.7]{gualdi_hypersurfaces_in_toric_varieties}
    Let $f$ be a nonzero Laurent polynomial over $K$. The Ronkin function of $f$ over a place $v$ is the map $\rho_f:N_\bbR \to \bbR$ defined by
    \[
        \rho_f: u \mapsto \int_{\trop^{-1}(u)} -\log| f(x)| d \sigma_u(x).
    \]
\end{definition}

By \cite[Proposition 2.10]{gualdi_hypersurfaces_in_toric_varieties}, $\rho_f$ is a concave continuous function with bounded difference from $\Psi_{NP(f)}$. By Theorem \ref{metrics_on_toric_divisors}, it defines a semipositive Green's function for $D(f)$. The collection of $v$-adic Ronkin functions gives rise to a Zhang divisor $R_f$ by \cite[Lemma 5.11]{gualdi_hypersurfaces_in_toric_varieties}.

\begin{theorem}(variation of \cite[Theorem 5.12]{gualdi_hypersurfaces_in_toric_varieties})
    Let $X_\Sigma$ be a proper toric variety. Let $f$ be a Laurent polynomial with vanishing locus $Z$ such that $NP(f)$ defines a divisor on $X_\Sigma$. Let $\ov{D}_0,\dots,\ov{D}_{n-1}$ be semipositive toric Zhang divisors on $X_\Sigma$. Then,
    \[
        \adeg(\ov{D}_0,\dots,\ov{D}_{n-1}|Z) = \adeg(\ov{D}_0,\dots,\ov{D}_{n-1},R_f|X_\Sigma).
    \]
\end{theorem}

For future use, we will relate the Ronkin functions and the Ronkin roof function of Laurent polynomials related by maps of tori.
\begin{definition}
    Let $\gamma: V \to W$ be a homomorphism of finite dimensional real vector spaces and $f:V\to\bbR\cup\{-\infty\}$ be a closed concave function with compact support. We define the direct image of $f$ along $\gamma$ by
    \[
    \gamma_* f(w)= \max_{v \in \gamma^{-1}(w)} f(v).
    \]
    It is a closed concave function with compact domain in $W$.
\end{definition}

\begin{lemma}\label{lemm:pushforward_of_ronkin}
    Let $\gamma:M \to M'$ be a map of lattices. Then, this induces a map on the rings of Laurent polynomials $K[\gamma]:K[M] \to K[M']$. Let $f \in K[M]\setminus \{0\}$. Then,
    \[
    \rho_{K[\gamma](f)} = \rho_f \circ \gamma^\vee
    \]
    and
    \[
    \rho_{K[\gamma](f)}^\vee = \gamma_* \rho_f^\vee(m).
    \]
\end{lemma}

\begin{proof}
    The first statement follows readily from the definitions. The second statement follows from the first using \cite[Proposition 2.3.8(1)]{burgos_philippon_sombra_toric_varieties}.
\end{proof}

We now prove generalizations of Lemma 1.11 and Proposition 1.12 in \cite{gualdi_hypersurfaces_in_toric_varieties}. We apply this to write the formulas in the preceding section in the form used in \cite[Conjecture 1]{gualdi:tel-01931089}.

\begin{lemma}
    Let $\Delta_1, \dots, \Delta_k$ be polytopes contained in a $k$-dimensional rational subspace $L$ and denote by $\pi$ the projection away from this subspace to the quotient $P$. Let $Q_1, \dots, Q_{n-k}$ be polytopes in $M_\bbR$. Then,
    \[
    \MV_M(\Delta_1, \dots, \Delta_k,Q_1, \dots, Q_{n-k}) = \MV_L(\Delta_1, \dots, \Delta_k)\cdot \MV_P(\pi(Q_1), \dots, \pi(Q_{n-k})).
    \]
\end{lemma}

\begin{proof}
    By the definition of mixed volume one obtains
\begin{align*}
    &\MV_M(\Delta_1, \dots, \Delta_k,Q_1, \dots, Q_{n-k})\\
    &=\sum_{j=1}^{n-k} (-1)^{n-j}\sum_{1\leq i_1<\dots<i_j\leq n-k}\sum_{I\subset \{1,\dots,k\}}(-1)^{|I|}\vol(\Delta_I + Q_{i_1}+\dots+Q_{i_j}).
\end{align*}
Here $\Delta_I$ is taken to denote $\sum_{i\in I}\Delta_i$. It now suffices to show that 
\[\sum_{I\subset \{1,\dots,k\}}(-1)^{|I|}\vol(\Delta_I + Q) = (-1)^k \MV_L(\Delta_1, \dots, \Delta_k)\vol(\pi(Q)).\]
For this take any $p \in P$ and denote by $Q_p$ the preimage of $p$ in $Q$. We may view $\sum_{I\subset \{1,\dots,k\}}(-1)^{|I|}\vol(\Delta_I + Q)$ as an integral over $\pi(Q)$, namely as
\[
\int_{\pi(Q)} \sum_{I\subset \{1,\dots,k\}}(-1)^{|I|}\vol(\Delta_I + Q_p).
\]
In order to conclude, we need to show that 
$\sum_{I\subset \{1,\dots,k\}}(-1)^{|I|}\vol(\Delta_I + R) = \MV_L(\Delta_1, \dots, \Delta_k)$ for any polytope $R$ in the k-dimensional subspace spanned by the $\Delta_i$. For this we decompose the expression by writing out the volumes as mixed volumes and order by the number of $R$ occurring in the expansion.
\begin{align*}
    &\sum_{I\subset \{1,\dots,k\}}(-1)^{|I|}\vol(\Delta_I + R)\\ &= 
\sum_{s=0}^{k} {\binom{k}{s}}\sum_{I\subset \{1,\dots,k\}}(-1)^{|I|} \sum_{1\leq j_1 \leq\dots\leq j_{k-s}\leq k, j_i \in I} {\binom{k-s}{J}}\MV(\Delta_{j_1},\dots,\Delta_{j_{k-s}},R,\dots,R).
\end{align*}

Here $\binom{k-s}{J}$ is taken to denote the number of partitions of $k-s$ elements into partitions of form $J$, i.e. the multinomial coefficient for $k-s$ and $\#\{i|j_i=m\}$. We reorder the sum to sum over size $k-s$ multisets of $\{1,\dots,k\}$. Fix a multiset $1\leq j_1 \leq\dots\leq j_{k-s}\leq k$ of order $k-s$ in elements of $\{1,\dots k\}$. Then, its contribution to the sum is 
\[
\binom{k}{s}\binom{k-s}{J}\MV(\Delta_{j_1},\dots,\Delta_{j_{k-s}},R,\dots,R) \sum_{I\supseteq J} (-1)^{|I|},
\]
where containment is understood on the level of underlying sets. By comparing to the expansion of $\prod_{I\setminus J}(1-1)$ we see that this vanishes for $J\neq I$ and is $(-1)^k$ for $J=I$.
\end{proof}

\begin{lemma}
    \label{lemm:convex_pushforward}
    Let the $g_i$ be concave functions on polytopes $Q_i$ and $\Delta_i$ as before. Then,
    \[
\MI_M(\iota_{\Delta_1},\dots,\iota_{\Delta_k},g_1,\dots,g_{n-k+1}) = \MV_L(\Delta_1, \dots, \Delta_k)\cdot \MI_P(\pi_*g_1, \dots, \pi_*g_{n-k+1}).
    \]
\end{lemma}
\begin{proof}
    The reduction to the previous Lemma is precisely as in \cite[Proposition 1.12]{gualdi_hypersurfaces_in_toric_varieties}.
\end{proof}

\section{Intersection product}~\label{sec:int_product}

In this section we study the intersection product defined in \cite{Adelic_curves_2} and prove that it varies continuously in flat families. We will treat the case of a trivial GVF in a separate section as the natural class of line-bundles differs from the non-trivial setting. Our main result is the following theorem. 
\begin{theorem}\label{thm:continuity_of_intersection_number}
    Let $\cX \to S$  be a flat projective morphism of finite type schemes over a GVF $K$ of relative dimension $d$. Let $\ov{\sL}_0,\dots,\ov{\sL}_d\in \intPicQ(\cX/S)$ be globally integrable line bundles on $\cX$ over $S$. Then, the map
    \begin{align*}
        S_{\GVF} &\to \bbR\\
        s&\mapsto\adeg(\ov{\sL}_0,\dots,\ov{\sL}_d|\sX_s)
    \end{align*}
    is continuous.
\end{theorem}

This can be applied in various settings of interest to number theorists. This can be seen for instance by the following corollary.

\begin{corollary}
    Suppose that $S$ is projective over $\bbQ$ and $\cX \to S$ a projective morphism. Then, any integrable Zhang line bundle on $\sX$ is a globally integrable line bundle in $\intPicQ(\cX/S)$. 
\end{corollary}

\begin{proof}
    It is a weaker property to be an element of $\intPicQ(\cX/S)$ than of $\intPicQ(\cX)$. Hence, the statement follows from Theorem \ref{theorem:Zhang_divisors_are_induced_by_global_divisors}.
\end{proof}

\begin{remark}
    The property of being globally integrable is preserved under base change. In particular, one may apply Theorem \ref{thm:continuity_of_intersection_number} to integrable Zhang divisors on a projective $\cX$ after restricting to the flat locus of $\cX \to S$.
\end{remark}

\subsection{Lattice divisors}~\label{subsection:Lattice_divisors}
Let $K$ be a countable GVF. We can assume that it is represented by a proper adelic curve $(K, (\Omega, \cA, \nu), \phi)$ with $\Omega = M_{K}$, the trivial absolute value having zero mass, and the restriction of the measure to the archimedean places $\nu|_{\Omega_\infty}$ being supported at normalized valuations (i.e., satisfying $v(2)=-\log 2$). Moreover, whenever we consider a GVF extension $K \subset \Field$ in this subsection, we assume that $\Field$ is also countable and the GVF structure on $\Field$ is induced by an adelic curve structure with the same properties (see \cite[Section 9]{basics_of_gvfs}).

We recall ideas from the theory of adelic curves only briefly. We refer to~\cite{Adelic_curves_1, Adelic_curves_2} for details. Let us start by recalling a definition.

\begin{definition}
    Let $X$ be a finite type $K$-scheme with a line bundle $L$. A \textit{metric family} on $L$ is a family $\varphi = (\varphi_{\omega})_{\omega \in \Omega}$, where each $\varphi_{\omega}$ is a continuous metric on $L_{\omega}$ on $X_{\omega}^{\an}$. Here $X_{\omega}^{\an}$ is the Berkovich analytification of $X_\omega = X \otimes_K K_{\omega}$, where $K_{\omega}$ is the completion of $K$ with respect to the absolute value $\omega \in \Omega = M_K$. We call a pair $\ov{L} = (L,\varphi)$ a \textit{metrized line bundle} on $X$ over $K$.
\end{definition}

This is as in \cite[Definition 4.1.4]{Adelic_curves_2}, but we also allow non-projective $X$. Also, we naturally extend this definition to $\QQ$-line bundles. Similarly, we extend the definition of a \textit{Green function family} of a ($\QQ$-)Cartier divisor on $X$ from \cite[Definition 4.2.1]{Adelic_curves_2}, to a non-projective $X$, word for word. A ($\QQ$-)Cartier divisor with a Green function family is called a \textit{metrized divisor}.

Chen and Moriwaki introduce \emph{adelic line bundles} as a subset of metrized line bundles respecting the global nature of the adelic curve, see \cite[Definition 4.1.9]{Adelic_curves_2}. There are two conditions. Firstly, the variation of metrics along $\omega$ has to be measurable. This condition does not occur for number fields since their set of places is discrete. Finally, one needs a condition requiring that the family of metrics has finite distance to arising from a global model. This condition is referred to as being dominated. In order to obtain an intersection theory, one needs to demand the metrics at each place to be \emph{integrable}, i.e.\ the difference of \emph{semipositive} metrics. We follow the definition of semipositivity from \cite[Section 2.3]{Adelic_curves_1}. Note that this only allows for semipositive metrics on semiample line bundles.

\begin{definition}~\label{definition_Weil_metric_copy}
    Consider $\PP^n = \PP_K^n$ with coordinates $x_0, \dots, x_n$. We equip the anti-tautological line bundle $\cO(1)$ with two families of metrics $\varphi = (|\cdot|_{\phi_\omega})_{\omega \in \Omega}, \psi = (|\cdot|_{\psi_\omega})_{\omega \in \Omega}$ defined in the following way. For a section $s \in H^0(\PP^n, \cO(1))$ identified with a linear form $A$ we have
    \[ |s(z)|_{\varphi_\omega} := \frac{|A(z)|_{\omega}}{\max(|z_0|_{\omega}, \dots, |z_n|_{\omega})},\]
    for $z \in \PP_{\omega}^{n, \an}$ and any $\omega \in \Omega$.
    The metric $\psi$ is defined in the same way for non-Archimedean $\omega$, but for archimedean $\omega$ we set
    \[ |s(z)|_{\psi_\omega} := \frac{|A(z)|_{\omega}}{\sqrt{\sum_{i=0}^n |z_i|_{\omega}^2}}, \]
    for all $z \in \PP_{\omega}^{n,\an}$. We call $\varphi, \psi$ the Weil and the Fubini-Study metric respectively and use the notation $\ov{\cO(1)} = (\cO(1), \varphi), \ov{\cO(1)}^{\FS} = (\cO(1), \psi)$. We use the same notation for pullbacks of these adelic line bundles on $\PP_S^n$ for any finite type $K$-scheme $S$. Both the Weil and the Fubini-Study metric define semipositive adelic line bundles.
\end{definition}

For a place $\omega \in \Omega$ and Green's functions $\phi,\psi$ for a divisor $D$ we denote by $d_\omega(\phi,\psi)$ the sup-norm of $\phi-\psi$ and call it the local distance of $\phi$ and $\psi$. For different metrics on a divisor over an adelic curve we denote by $d(\phi,\psi)$ the global distance of $\phi$ and $\psi$. It is defined as the upper integral $\int^+d_\omega(\phi_\omega,\psi_\omega)\nu (d\omega)$ over the local distances (see \cite[Definition A.4.1]{Adelic_curves_1}).

From the point of view of globally valued fields the Weil metric is the fundamental object. We need to relate it to the Fubini-Study height to invoke calculations from that setting.

\begin{lemma}~\label{lemma_Veronese_pullbacks_converge_to_Weil}
    Let $\alpha_n : \PP^r \to \PP^s$ be the $n$-th Veronese map. Consider two metrics on the line bundle $\cO(1)$ on $\PP^r$, namely the Weil metric $\varphi$ and the metric $\sqrt[n]{\alpha_n^* \psi}$ which is the $n$-th root of the pullback of the Fubini-Study metric from $\cO(1)$ on $\PP^s$. Then for non-Archimedean places the two metrics are the same and for archimedean places $\sigma \in \Omega$ we have
    \[ d_{\sigma}(\varphi, \sqrt[n]{\alpha_n^* \psi}) \leq \frac{r \log n}{n}, \]
    for $n \geq r+1$.
\end{lemma}
\begin{proof}
    We only calculate the archimedean case. Consider the section $x_0$ of $\cO(1)$ on $\PP^r$. By definition we have
    \[ |x_0(z)|_{\varphi} = \frac{|z_0|_{\sigma}}{\max(|z_0|_{\sigma}, \dots, |z_n|_{\sigma})} \]
    and
    \[ |x_0(z)|_{\sqrt[n]{\alpha_n^* \psi}} = \sqrt[n]{\frac{|z_0|_{\sigma}^n}{\sqrt{\sum_{|I|=n} |z^I|_{\sigma}^2}}} = \frac{|z_0|_{\sigma}}{\sqrt[2n]{\sum_{|I|=n} |z^I|_{\sigma}^2}}. \]
    Hence we calculate
    \[ \bigg| -\log \frac{|x_0(z)|_{\varphi}}{|x_0(z)|_{\sqrt[n]{\alpha_n^* \psi}}} \bigg| =  \bigg| \log \frac{\max(|z_0|_{\sigma}, \dots, |z_n|_{\sigma})}{\sqrt[2n]{\sum_{|I|=n} |z^I|_{\sigma}^2}} \bigg|. \]
    This is bounded by 
    \[|\log \sqrt[2n]{s+1}| = \bigg| \log \sqrt[2n]{\binom{r+n}{n}} \bigg| \leq  \frac{1}{2n} |\log (r+1)n^r| \leq \frac{r \log n}{n}, \]
    where in the first inequality we use the fact that $\binom{r+n}{n} \leq (r+1)n^r$ and the second inequality holds for $n \geq r+1$.
\end{proof}

\begin{lemma}~\label{lemma_pullback_of_adelic_are_adelic_copy}
    Let $f:Q \to P$ be a morphism of projective schemes over $K$ and let $\ov{L}$ be an adelic line bundle on $P$. Then $f^*\ov{L}$ is an adelic line bundle on $Q$.

    If $\ov{L}$ is semipositive or integrable, so is $f^*\ov{L}$.
\end{lemma}
\begin{proof}
    The fact that $f^*\ov{L}$ is an adelic line bundle is found in \cite[Section 2.8.3 and 2.9.5]{chenpositivity}. The semipositivity assertion follows from \cite[Lemma 6.1.2]{chenpositivity}.
\end{proof}

Fix a projective morphism $\pi:\cX \to S$ of finite-type $K$-schemes, where $S$ is not necessarily projective.

\begin{definition}~\label{definition:simple_line_bundles}
    We say that a metrized family $\ov{\cL}$ on $\cX$ is \textit{simple over $S$}, if there is a closed embedding $j:\cX \to \PP_S^n$ over $S$ (for some $n$), such that $\ov{\cL} = j^*\ov{\cO(1)}$. The elements of the $\QQ$-vector space of metrized $\QQ$-line bundles on $\cX$ generated by simple ones are called \textit{lattice line bundles on $\cX$ over $S$} and are denoted by $\LPicQ(\cX/S)$. The space of metrized $\QQ$-divisors coming from rational sections of such metrized $\QQ$-line bundles is denoted by $\LdivQ(\cX/S)$ and its elements are called \textit{lattice divisors on $\cX$ over $S$}. If $S$ is equal to $\Spec(K)$, we omit it in the notation. Moreover, if we want to emphasize the dependence on the GVF $K$, we use the notation $\LPicQ(\cX/S)_K$.
\end{definition}

\begin{remark}
    It follows from Lemma~\ref{lemma_pullback_of_adelic_are_adelic_copy} that simple metrized line bundles on $\cX$ over $S$ are semipositive on fibers. More precisely, if $K \subset \Field$ is a GVF extension, and $s \in S(\Field)$, then a simple metrized line bundle $\ov{\cL}$ over $S$ coming from an embedding $j:\cX \to \PP_S^n$ induces an adelic line bundle $\ov{\cL}_s = j_s^*\ov{\cO(1)}$ which is semipositive on $\cX_s$ with respect to the GVF $\Field$.
\end{remark}

We remark that for any GVF extension $\Field/K$, a finite type $\Field$-scheme $T$, and a morphism of $K$-schemes $T \to S$, there is a base-change map
    \[ \LPicQ(\cX/S)_K \to \LPicQ(\cX_T/T)_{\Field}. \]
In particular for $s \in S(\Field)$, there is a specialisation map
    \[ \LPicQ(\cX/S)_K \to \LPicQ(\cX_s)_\Field. \]
Let us describe these maps more precisely. Let $\ov{\cL} = (\cL,\varphi) \in \LPicQ(\cX/S)_K$, for $\varphi = (\varphi_{\omega})_{\omega \in M_{K}}$. This means that each $\varphi_{\omega}$ is a metric on $\cL_{\omega}$ over $\cX_{\omega}^{\an}$. To get a family of metrics $\psi = (\psi_v)_{v \in M_{\Field}}$ on the base-change $\cL_T$ of $\cL$ via the morphism $\cX_T \to \cX$ one proceeds as in \cite[Example 4.1.8]{Adelic_curves_2} (which works the same when $K'=\Field/K$ is not algebraic). Equivalently, one could take an embedding $j:\cX \to \PP_S^n$ that realises $\ov{\cL} = j^*\ov{\cO(1)}$ (or two embeddings such that it comes from the difference of pullbacks) and define $(\cL_T,\psi)$ via the pullback of $\ov{\cO(1)}$ through the map $j_T:\cX_T \to \PP_T^n$.

\subsection{Heights of resultants}~\label{subsection_heights_of_resultants}
Chen and Moriwaki have constructed an intersection product of integrable adelic Cartier divisors in \cite{Adelic_curves_2}. In this subsection we look closely at its definition which uses heights of certain resultants.

\begin{theorem}~\cite[Theorem B]{Adelic_curves_2}~\label{theorem_adelic_intersection_product}
    Let $\Pvariety$ be a projective scheme of pure dimension $d$ over a GVF $K$. Then, there is a multilinear \textit{adelic intersection product}
    \[ \LPic_{\QQ}(\Pvariety)^{d+1} \to \RR.\]
\end{theorem}
\begin{proof}
    Given finitely many elements of $\LPic_{\QQ}(\Pvariety)$ we can replace $K$ by a countable subfield over which the corresponding embeddings to projective spaces (and $\Pvariety$) are defined. Then we can represent the GVF $K$ by an adelic curve. Since lattice line bundles are integrable, an arithmetic intersection number is defined by \cite[Theorem B]{Adelic_curves_2}. We show that the product on $\LPicQ(\Field)$ only depends on the induced GVF structure on $K$ in Corollary~\ref{corollary_intersection_product_only_depends_on_the_GVF_structure}.
\end{proof}

We write the intersection number of lattice line bundles $\ov{L}_0,\dots,\ov{L}_d$ as $\ov{L}_0\cdots\ov{L}_d$. We observe that the adelic intersection product is determined by its values on tuples of simple line bundles. Let us fix a tuple of such simple adelic Cartier divisors on  a projective scheme $\Pvariety$ and analyse how to calculate their intersection.

Assume we are given closed embeddings $\xi_i:X \to \PP^{r_i} = \PP(V_i)$ for $i=0, \dots, d$, where $V_i$ is a $(r_i+1)$-dimensional vector space over $K$ with a distinguished basis. For a natural number $n$, denote by $\xi_i^{\otimes n} : X \to \PP^{r_i(n)}$ the composition of $\xi_i$ with the $n$-th Veronese map $\PP^{r_i} = \PP(V_i) \to \PP(S^{n} V_i) = \PP^{r_i(n)}$ and write $V_i(n) := S^{n} V_i$. Note that in this case we have
\[ \dim V_i(n) = r_i(n) + 1 = \binom{r_i+n}{n} = O(n^{r_i}). \]

For each $n$ we define the line bundle $L_i(n)$ to be the pullback $\xi_i^{\otimes n,*}\cO(1)$. We pull back the Weil metric and the Fubini-Study metric to obtain adelic line bundles $\ov{L_i(n)}$ and $\ov{L_i(n)}^{\FS}$ respectively. For $n=1$, we omit the $n$ in the notation. We note that there is a canonical isomorphism $L_i(n) \cong L_i^{\otimes n}$.

Let $\delta_i(n)$ be the intersection number
$L_0(n) \cdots L_{i-1}(n) \cdot L_{i+1}(n) \cdots L_d(n)$ and set $\delta_i = \delta_i(1)$. Since $L_i(n) \cong L_i^{\otimes n}$, we have $\delta_i(n) = n^d \cdot \delta_i$. Let
\[ W(n) :=  S^{\delta_0(n)}(V_0(n)^{\vee}) \otimes_K \ldots \otimes_K S^{\delta_d(n)}(V_d(n)^{\vee})\]
and note that using distinguished bases of $V_i$ for $i=0, \dots, d$ we can naturally interpret elements of $W(n)$ as polynomials of multi-degree $(\delta_0(n), \dots, \delta_d(n))$ on $V_0(n) \times \dots \times V_d(n)$. There is a unique (up to scaling) element $R_n \in W(n)$ such that it vanishes on $(v_0, \dots, v_d)$ if and only if the intersection $X \cap Z(v_0) \cap \dots \cap Z(v_d)$ is non-empty (as a scheme), where $Z(v_i)$ is the pullback to $X$ of the hyperplane in $\bbP(V_i(n))$ defined by the zero-set of the linear form $v_i$. We call it the resultant of $X$ with respect to embeddings $\xi_i^{\otimes n}$. It determines a unique element $R_n \in \bbP(W(n))$ whose height calculates the adelic intersection product in the following way. 
\begin{remark}~\label{remark_Chen_Moriwaki_formula_intersections_and_resultants}~\cite[Remark 4.2.13]{Adelic_curves_2}
        \[ \ov{L_0(n)}^{\FS} \cdot \ldots \cdot \ov{L_d(n)}^{\FS} = \int_{\Omega \setminus \Omega_{\infty}} \log \|R_n\|_{\omega} \nu(d\omega)\]
        \[+ \int_{\Omega_{\infty}} \nu(d\sigma) \int_{\bbS_{0}(n)_{\sigma} \times \dots \times \bbS_{d}(n)_{\sigma}} \log |(R_n)_{\sigma}(z_0, \dots z_d)| \eta_{\bbS_{0}(n)_{\sigma}}(dz_0) \otimes \dots \otimes \eta_{\bbS_{d}(n)_{\sigma}}(dz_d)\]
        \[+ \nu(\Omega_\infty) \frac{1}{2} \sum_{i=0}^{d} \delta_i(n) \sum_{l=1}^{r_i(n)} \frac{1}{l},\]
where we use the notation from the cited remark, but with an additional variable $n$. This means that $\bbS_{i}(n)_{\sigma}$ is the unit sphere in $V_i(n)_{\sigma}$ with the sphere measure $\eta_{\bbS_{i}(n)_{\sigma}}$. Moreover, for a non-Archimedean $\omega \in \Omega$, the norm $\|\cdot\|_{\omega}$ is the maximum of coefficients norm, with respect to the distinguished basis of $W(n)$, for example by \cite[Proposition A.2.2]{chenpositivity}. Later we write $\eta(dz)$ for $\eta_{\bbS_{0}(n)_{\sigma}}(dz_0) \otimes \dots \otimes \eta_{\bbS_{d}(n)_{\sigma}}(dz_d)$ and $z$ for the tuple $z_0, \dots, z_d$ (here $n, d$ and $\sigma$ are implicit).
\end{remark}

\begin{lemma}~\label{lemma_aproximating_intersection_of_tropical_by_Fubini_Study}
    The adelic intersection product satisfies
    \[ \lim_n \frac{1}{n^{d+1}} \ov{L_0(n)}^{\FS} \cdot \ldots \cdot \ov{L_d(n)}^{\FS} = \ov{L_0(n)} \cdot \ldots \cdot \ov{L_d(n)}. \]
\end{lemma}
\begin{proof}
    It suffices to show that the metrics on $\frac{1}{n} \ov{L_i(n)}^{\FS}$ converge with respect to the global distance to the metrics on $\ov{L_i}$. By Lemma~\ref{lemma_Veronese_pullbacks_converge_to_Weil} the global distance satisfies $d(\frac{1}{n} \ov{L_i(n)}^{\FS},\ov{L_i}) \leq \nu(\Omega_\infty) \cdot \frac{r_i\log n}{n}$.
\end{proof}

We use the above formula and the lemma to calculate $\ov{D}_0 \cdot \ldots \cdot \ov{D}_d$ through resultants. More precisely we show the following.
\begin{proposition}~\label{proposition_formula_for_adelic_intersection_as_heights_of_resultants}
    In the above context we have
    \[ \ov{L}_0 \cdot \ldots \cdot \ov{L}_d = \lim_n \frac{1}{n^{d+1}} \height(R_n), \]
    where we treat $R_n$'s as tuples, using the distinguished basis of $W(n)$.
\end{proposition}
\begin{proof}
By Lemma~\ref{lemma_aproximating_intersection_of_tropical_by_Fubini_Study} we only need to show that
\[ |\ov{L_0(n)}^{\FS} \cdot \ldots \cdot \ov{L_d(n)}^{\FS} - \height(R_n)| = o(n^{d+1}). \]
First we express $\height(R_n)$ in a form of an integral
\[ \height(R_n) = \int_{\Omega} \log \|R_n\|_{\omega} \nu(d\omega) \]
\[ = \int_{\Omega \setminus \Omega_{\infty}} \log \|R_n\|_{\omega} \nu(d\omega) + \int_{\Omega_{\infty}} \log \|R_n\|_{\sigma} \nu(d\sigma), \]
where $\|\cdot\|_{\omega}, \|\cdot\|_{\sigma}$ denote the maximum of coefficients norms (for a non-Archimedean $\omega \in \Omega$ or archimedean $\sigma \in \Omega_{\infty}$), with respect to the distinguished basis of $W(n)$.
\begin{claim}
    We have
    \[ \nu(\Omega_\infty) \frac{1}{2} \sum_{i=0}^{d} \delta_i(n) \sum_{l=1}^{r_i(n)} \frac{1}{l} = O(n^{d} \log n). \]
    In particular, when divided by $n^{d+1}$ converges to zero, for $n \to \infty$.
\end{claim}
\begin{proof}
    This follows from the fact that $\delta_i(n) = n^d \delta_i$ and 
    \[ \sum_{l=1}^{r_i(n)} \frac{1}{l} = O(\log r_i(n)) = O(\log n^{r_i}) = O(\log n). \]
\end{proof}
By Remark~\ref{remark_Chen_Moriwaki_formula_intersections_and_resultants} we will be done if we show that
\[ \Big| \int_{\bbS_{0}(n)_{\sigma} \times \dots \times \bbS_{d}(n)_{\sigma}} \log |(R_n)_{\sigma}(z)| \eta(dz) - \log \|R_n\|_{\sigma} \Big| = o(n^{d+1}), \]
where the constant is independent of $\sigma \in \Omega_{\infty}$. But by Proposition~\ref{proposition:comparison_fsmahler_norm_multi} applied to the polynomial $(R_n)_{\sigma}$, we have
\[
    \Big| \int_{\bbS_{0}(n)_{\sigma} \times \dots \times \bbS_{d}(n)_{\sigma}} \log |(R_n)_{\sigma}(z)| \eta(dz) - \log \|R_n\|_{\sigma} \Big| \leqslant \sum_{i=0}^d \delta_i(n) \left(\log(r_i(n)+1) + \sum_{k=1}^{r_i(n)-1}\frac{1}{k}\right),
\]
where $\delta_i(n) = n^d\delta_i$ and $\log(r_i(n)+1) + \sum_{k=1}^{r_i(n)-1}\frac{1}{k} = O(\log(r_i(n))) = O(\log n)$ for all $i\leqslant d$, so
\[
    \sum\limits_{i=0}^d \delta_i(n) \left(\log(r_i(n)+1) + \sum_{k=1}^{r_i(n)-1}\frac{1}{k}\right) = O(n^d\log n) = o(n^{d+1}),
\]
where the given bound does not depend on $\sigma$.
\end{proof}

\begin{corollary}~\label{corollary_intersection_product_only_depends_on_the_GVF_structure}
    The intersection product on $\LPicQ(\Pvariety)$ only depends on the induced GVF structure on $K$.
\end{corollary}

\begin{remark}~\label{remark_uniformity_in_height_of_resultants_formula}
    By analysing precisely the proof of Proposition~\ref{proposition_formula_for_adelic_intersection_as_heights_of_resultants}, one can see that in fact it shows existence of an absolute constant $C$ such that
    \[ \Big| \ov{L}_0 \cdot \ldots \cdot \ov{L}_d - \frac{1}{n^{d+1}} \height(R_n) \Big| \leq C \cdot (1+\nu(\Omega_{\infty})) \cdot \max_i r_i\delta_i \cdot \frac{\log n}{n}. \]
    Note that the number $\nu(\Omega_{\infty})$ can be expressed as $\frac{\height(2)}{\log 2}$ with respect to the induced GVF structure on $K$.
\end{remark}

\subsection{\label{sec:definability}Definability of adelic intersection product}

In this section we prove Theorem \ref{thm:continuity_of_intersection_number} for lattice line bundles. Let $\pi:\cX \to S$ be a flat projective morphism with $d$-dimensional fibers. Suppose $S = \Spec A$ is an affine variety (so that $A$ is an integral domain).

Let $\ov{\cL}_0, \dots, \ov{\cL}_d \in \LPic(\cX/S)$ be simple over $S$ with embeddings $\alpha_i:\cX \to \PP^{k_i}_S$. For a field-valued point $s \in S(\Field)$ and any object $Q$ over $S$, the notation $Q(s)$ denotes its base change to $s$. Denote by $\delta_i$ the intersection number
\[ \deg(\cL_0(s) \cdot \ldots \cdot \cL_{i-1}(s) \cdot \cL_{i+1}(s) \cdot \ldots \cdot \cL_d(s)|\cX_s)\]
for any $s \in S(\Field)$ in any field extension $K \subset \Field$. It is independent of the choice of $s$ since the family $\pi:\cX \to S$ is flat. 

\begin{lemma}~\label{lemma:uniform_resultants}
    There is a family of polynomials $R_n$ with coefficients in $A$ defined up to scalar in $A^\times$, such that for all $s \in S(\Field)$ we have
    \[ R_n(s) = R_{n,s}, \]
    where $R_{n,s}$ is the resultant $R_n$ from Subsection~\ref{subsection_heights_of_resultants}, defined for the scheme $\cX_s$ equipped with the family of embeddings $(\alpha_i)_s:\cX_s \to \PP_{\kappa(s)}^{k_i}$ for $i=0, \dots, d$.
\end{lemma}

\begin{proof}
    This is probably standard, but we could not find a reference so we sketch a proof here, based on the construction of resultants from \cite[Section 1.6]{Adelic_curves_2}. We use notation from loc.cit but with the base-field $k$ replaced by $A$, $X$ over $k$ replaced by $\cX$ over $S = \Spec(A)$, and $L_i$'s replaced by $\cL_i$'s.

    First note that the whole Section 1.5 and Section 1.6 up to Proposition 1.6.2 of \cite{Adelic_curves_2} go through word-for-word over $A$. It remains to prove the analogue of \cite[Proposition 1.6.2]{Adelic_curves_2} over $A$. This boils down to calculating the cycles
    \[  q_*(c_1(p^*\cL_{0}) \cdots c_1(p^*\cL_{i-1}) c_1(q^*q_i^*(\cO_{E_i^{\vee}}(1))) c_1(p^*\cL_{i+1}) \cdots c_1(p^*\cL_{d}) \cap [\cX \times_S \check{\bbP}])
    \]
    \[
    = c_1(q_i^*(\cO_{E_i^{\vee}}(1))) \cdot q_*(c_1(p^*\cL_{0}) \cdots c_1(p^*\cL_{i-1}) c_1(p^*\cL_{i+1}) \cdots c_1(p^*\cL_{d}) \cap [\cX \times_S \check{\bbP}])
    \]
    We look at the diagram
    \[\begin{tikzcd}
    	{\cX \times_S \check{\bbP}} & \cX \\
    	{\check{\bbP}} & {S}
    	\arrow["p"', from=1-1, to=1-2]
    	\arrow["q", from=1-1, to=2-1]
    	\arrow["r"', from=1-2, to=2-2]
    	\arrow["s", from=2-1, to=2-2]
    \end{tikzcd}\]
    and use flat base-change to get the equality
    \[ q_*(c_1(p^*\cL_{0}) \cdots c_1(p^*\cL_{i-1}) c_1(p^*\cL_{i+1}) \cdots c_1(p^*\cL_{d}) \cap [\cX \times_S \check{\bbP}]) \]
    \[ = s^*r_*(c_1(\cL_{0}) \cdots c_1(\cL_{i-1})c_1(\cL_{i+1}) \cdots c_1(\cL_{d}) \cap [\cX]). \]
    Let $\eta$ be the generic point of $S$. By the flat base-change for the localisation map $\eta \to S$, we get
    \[ r_*(c_1(\cL_{0}) \cdots c_1(\cL_{i-1})c_1(\cL_{i+1}) \cdots c_1(\cL_{d}) \cap [\cX]) 
    \]
    \[ = \deg(\cL_{0}(\eta) \cdots \cL_{i-1} (\eta) \cL_{i+1}(\eta) \cdots \cL_{d}(\eta)) [S] \]
    which is equal to $\delta_i \cdot [S]$. Hence the cycle in question is equal to 
    \[ c_1(q_i^*(\cO_{E_i^{\vee}}(1))) s^*(\delta_i \cdot [S]) = c_1(q_i^*(\cO_{E_i^{\vee}}(1))) \cap \delta_i [\check{\bbP}] = c_1(q_i^*(\cO_{E_i^{\vee}}(\delta_i))) \cap [\check{\bbP}],\]
    which finishes the proof as in \cite[Proposition 1.6.2]{Adelic_curves_2}.
    
\end{proof}

\begin{proposition}~\label{proposition_definability_for_simple_famielies}
    Let $\cX \to S$ be a flat projective morphism of finite type schemes over a GVF $K$ of relative dimension $d$. Let $\ov{\cL}_0, \dots, \ov{\cL}_d \in \LPic(\cX/S)$ be lattice line bundles on $X$ over $S$. Then, the map
    \begin{align*}
        S_{\GVF} &\to \bbR\\
        s&\mapsto\adeg(\ov{\cL}_0(s), \dots, \ov{\cL}_d(s)|\sX_s)
    \end{align*}
    is continuous.
\end{proposition}
\begin{proof}
    By Remark~\ref{remark:GVF_analytification_as_qf_GVF_types} (continuity with respect to the constructible topology) we may without loss of generality assume that $S$ is an affine variety and the line bundles are simple over the base $S$.
    
    Fix a net $(s_i)_{i} \in S_{\GVF}$ and $s \in S_{\GVF}$ such that $s_i \to s$. Using the notation from Lemma~\ref{lemma:uniform_resultants}, put 
    \[ I_{n, i} = \frac{1}{n^{d+1}} \height(R_n(s_i)), \quad I_n = \frac{1}{n^{d+1}} \height(R_n(s)),\]
    \[ I_i = \lim_n \frac{1}{n^{d+1}} \height(R_n(s_i)), \quad I = \lim_n \frac{1}{n^{d+1}} \height(R_n(s)).\]

    We need to show that $\lim_i I_i = I$. Pick a positive number $\varepsilon$. First, note that there is a natural number $n$ such that
    \[ |I - I_n| < \varepsilon \]
    and
    \[ |I_{i} - I_{n,i}| < \varepsilon \]
    for all $i$. Indeed, this is possible because by Remark~\ref{remark_uniformity_in_height_of_resultants_formula} the above differences are bounded by $\frac{\log n}{n}$ times an absolute multiple of $(1+\nu(\Omega_\infty)) \max_i k_i \delta_i$. Next, note that for $i$ big enough we have
    \[ |I_{n,i} - I_n| < \varepsilon \]
    because for a fixed $n$, $\height(R_n(y))$ is a continuous function on $S_{\GVF}$. Hence together we get
    \[ |I - I_{i}| < 3 \varepsilon, \]
    which finishes the proof as $\varepsilon>0$ was arbitrary.
\end{proof}

\subsection{Integrable divisors over globally valued fields}

For applications, it is useful to consider not only lattice line bundles, but also to allow certain limit metrics. We will define global line bundles and globally semipositive line bundles over a GVF in the spirit of Zhang line bundles. Let $\cX \to S$ be a projective morphism of finite type schemes over a GVF $K$.

\begin{definition}
    A \emph{lattice line bundle} $\ov{\cL}$ on $\cX$ over $S$ is called \textit{semipositive} if for every $s\in S_{\GVF}$ the family of metrics $\varphi|_{\cX_s}$ consists of semipositive metrics over almost all places $\omega \in M_{\kappa(s)}$.
\end{definition}

\begin{definition}
    Let $\cL$ be a line bundle on $\cX$ over $S$. For every compact set $C \subset S_{\GVF}$, we define a pseudometric $d_C$ on the space of metrics on $\cL$. Let $s$ be a section of $\cL$ and let $\phi$ and $\psi$ be two families of metrics on $\cL$. Then, we define
    \[
        d_C(\phi,\psi) =  \sup_{z\in C}\int^+_{\Omega_{\kappa(z)}} \sup_{x \in \cX_{s,\omega}^{\an}} |\log |s|_\phi - \log |s|_\psi| \nu(d\omega),
    \]
    where $\int^+$ is the upper integral functional defined in \cite[Definition A.4.1]{Adelic_curves_1}.
\end{definition}

Note that the set of (semipositive) adelic divisors is closed under this norm. The norm allows us to define notions of global and integrable divisors on $\cX$ over $S$.

\begin{definition}~\label{definition:global_line_bundles}
    A \emph{global line bundle} $\ov{\sL}$ on $\sX$ over $S$ is defined to be a line bundle $\sL$ with a metric family $\phi$ such that there is a sequence of lattice line bundles $\ov{\sL}_k=(\sL,\phi_k)$ such that
    \[
        \lim_{k \to\infty} d_C(\phi_k,\phi) = 0
    \]
    for every compact $C \subset S_{\GVF}$.

    It is called globally semipositive if the metric families $\phi_k$ can be chosen semipositive. A global line bundle $\ov{\sL}$ is called globally integrable if there are globally semipositive line bundles $\ov{\sL}_+$ and $\ov{\sL}_-$ and an isometry $\ov{\sL} \cong \ov{\sL}_+ \otimes \ov{\sL}^\vee_-$.

    We denote the isometry classes of global $\bbQ$-line bundles by $\aPicQ(\cX/S)$. The subgroup of integrable line bundles is denoted by $\intPicQ(\cX/S)$. If $S = \Spec K$, we often omit it in the notation. We furthermore denote the distance between two metric families by $d$.
\end{definition}

\begin{proposition}~\label{proposition:intersection_product_over_a_GVF}
    Let $\Pvariety$ be a projective scheme of pure dimension $d$ over a countable GVF $K$. The intersection product on lattice divisors extends to a pairing
    \[
        \intPicQ(\Pvariety)^{d+1} \to \bbR.
    \]
\end{proposition}

\begin{proof}
    By linearity, it suffices to construct the pairing for globally semipositive divisors. It suffices to show that on the set of semipositive lattice divisors the intersection product is continuous with respect to $d$.
    
    Let $\ov{L}=(\sO,\phi)$ be a lattice line bundle with trivial underlying line bundle and $d(\phi,0) = C$ and let $\ov{L}_1,\dots,\ov{L}_d \in \LPicQ^+(\Pvariety)$ be semipositive lattice line bundles. Then,
    \[
        |\ov{L} \cdot \ov{L}_1\cdots \ov{L}_d| \leq C \deg(L_1\cdots L_d).
    \]
    This can be read off from the interpretation of the intersection number as the integral over local heights
    \begin{align*}
        |\ov{L} \cdot \ov{L}_1\cdots \ov{L}_d| &= |\int_\Omega \int_{\Pvariety_{\omega}^{\an}} \log|1|_{\phi,\omega} \ c_1(\ov{L}_{1,\omega})\cdots c_1(\ov{L}_{d, \omega}) \nu(d\omega)|\\ &\leq \int^+_\Omega \sup_{x \in \Pvariety_{\omega}^{\an}}|\log|1|_{\phi,\omega}(x)|\deg(\ov{L}_1\cdots \ov{L}_d)\nu(d\omega)\\ &\leq C \deg(L_1\cdots L_d).
    \end{align*}
\end{proof}

We are finally in the position to prove Theorem \ref{thm:continuity_of_intersection_number}. We restate it for convenience.

\begin{theorem}\label{thm:continuity_approximation}
    Let $\cX \to S$ be a flat projective morphism of finite type schemes over a GVF $K$ of relative dimension $d$. Let $\ov{\sL}_0,\dots,\ov{\sL}_d\in \intPicQ(\cX/S)$ be globally integrable divisors on $\cX$ over $S$. Then, the map
    \begin{align*}
        S_{\GVF} &\to \bbR\\
        s&\mapsto\adeg(\ov{\sL}_0,\dots,\ov{\sL}_d|\sX_s)
    \end{align*}
    is continuous.
\end{theorem}

\begin{proof}
    We assume by linearity that $\ov{\sL}_0,\dots,\ov{\sL}_d$ are all semipositive. Let $C\subset S_{\GVF}$ be a compact subset. For $k\in \bbN$, let $\ov{\sL}^k_0,\dots,\ov{\sL}^k_d$ be semipositive lattice line bundles such that $d_C(\ov{\sL}^k_i,\ov{\sL}_i)$ converges to zero. Then, the functions $\adeg(\ov{\sL}^k_0,\dots,\ov{\sL}^k_d|\sX_s)$ converge uniformly to $\adeg(\ov{\sL}_0,\dots,\ov{\sL}_d|\sX_s)$ on $C$. Since the former are continuous by Proposition \ref{proposition_definability_for_simple_famielies}, the latter is, too. We are done since $S_{\GVF}$ is locally compact by remark \ref{remark:GVF_analytification_as_qf_GVF_types}.
\end{proof}

For applications the following theorem is crucial. It follows from \cite{Zhang_thesis_inequality}, but we have not found a suitable reference for the precise statement we need. Our reasoning uses some techniques from the proof of the arithmetic Demailly theorem in \cite{qu2023arithmeticdemaillyapproximationtheorem} or from \cite{Charles2021ArithmeticAA}. Note that we have the notion of semipositivity and integrability for both Zhang line bundles and global line bundles and that they are a priori different.

\begin{theorem}~\label{theorem:Zhang_divisors_are_induced_by_global_divisors}
    Every integrable Zhang line bundle $\ov{L}=(L,\phi)$ on a projective variety over a number field is induced by an integrable global divisor.
\end{theorem}
\begin{proof}
We prove that arithmetically ample divisors are induced by integrable global divisors. Arithmetically ample divisors in turn are dense in semipositive Zhang divisors allowing us to finish the proof.

\begin{definition}
    A hermitian line bundle $\ov{\sL}$ over an arithmetic variety $\sX \to \Spec \bbZ$ is called arithmetically ample if
    \begin{enumerate}
        \item $\sL_\bbQ$ is ample,
        \item the metrics on $\ov{\sL}$ are semipositive at each place,
        \item the height $\ac_1(\ov{\sL}|_{\sY})^{\dim\sY}> 0$ for every irreducible horizontal subvariety $\sY \subseteq \sX$.
    \end{enumerate}
    A Zhang divisor is called arithmetically ample if it is induced by an arithmetically ample Hermitian line bundle.
\end{definition}

We want to approximate an arithmetically ample arithmetic divisor $\ov{L}$ defined on the model $\sX$ from below. For this we apply the maps $\iota_n:\cX \to \bbP(H^0(nL))$. We endow the line bundle $\cO(1)$ on $\bbP(H^0(nL))$ with the metric $h_n$ induced by the supremum norm on $H^0(nL)$. Semipositivity implies precisely that the induced metrics on $L$ converge uniformly to $\ov{L}$ by \cite[Theorem 3.5]{Zhang_thesis_inequality}. This is always semipositive at all places and the underlying line bundle is ample. It is arithmetically ample since the subspace of integral sections in $H^0(nL)$ has a basis $s_1,\dots,s_N$ consisting of strictly small integral sections, at least for $n$ large enough. We are reduced to proving the claim for $(\cO(1), h_n)$ on $\bbP(H^0(nL))$. At finite places, the metric on $\cO(1)$ agrees with the Weil metric for the basis $s_1\dots,s_N$ (as in \cite[Claim 3.1.15]{szachniewicz2023}). It remains to show the approximation at the infinite place. We can approximate the metric $h_n$ by a smooth metric with everywhere positive curvature.

From now on we assume that $L$ induces $\cO(1)$ on some projective space $\bbP^n$ with Weil metrics at finite places and a smooth metric with everywhere positive curvature at $\infty$. By Dini's theorem it suffices to prove pointwise approximation, i.e. for every point $x \in X(\bbC)$ and $\epsilon > 0$ there exists an integer $N > 0$ and a small integral section $s \in H^0(\cO(N))$ such that $-\frac{1}{N}\log|s(x)| < \epsilon$.

We prove first that for arbitrarily big $N$ we can find $l\in H^0(\bbP^n_{\bbR},\cO(N))$ satisfying $-\log|l|_{\sup} \geq \epsilon$ and $-\frac{1}{N}\log|l(x)| < 2 \epsilon$. Let $\ov{x}$ be the complex conjugate of $x$. We apply \cite[Theorem 2.2]{Zhang_thesis_inequality} to $\cO(1)$ and $Y=\{x,\ov{x}\}$ to obtain a holomorphic section $s$ of $\cO(N)$ with $-\log|s|_{\sup} = 0$ and $-\frac{1}{N}\log|s_N(x)|,-\frac{1}{N}\log|s_N(\ov{x})| < \epsilon/2$. The section $s_N \otimes \ov{s_N} \in H^0(\bbP^n_\bbC,\cO(2N))$ is then a section $l\in H^0(\bbP^n_{\bbR},\cO(2N))$ satisfying $-\log|l|_{\sup} \geq 0$ and $-\frac{1}{2N}\log|l(x)| < \epsilon$. Rescaling proves the claim.

The vector space $H^0(\bbP^n_{\bbR},\cO(N))$ has a norm given by the supremum norm on $X(\bbC)$. The global sections over $\bbZ$ form a lattice $\Lambda_N = H^0(\bbP^n_{\bbZ},\cO(N))$. By \cite[Theorem 4.2]{Zhang_thesis_inequality}, there exists $0<r<1$ such that for big enough $N$, there is a basis of $\Lambda_N$ consisting of vectors of norm $<r^N$. For $r<r'<1$, and big enough $N$ it follows that for every $l \in H^0(\bbP^n_{\bbR},\cO(N))$ there exists $l' \in \Lambda_N$ with $|l-l'|_{\sup} < (r')^N$. We apply this to the section $l$ constructed in the previous paragraph. Then, $l'$ eventually satisfies $|l'|_{\sup} \leq 1$. Furthermore, for small enough $\epsilon$ in the construction of $l$ we can ensure $-\log|r'| > 2 \epsilon$. Then, for big enough $N$ we have $-\frac{1}{N}\log|l'(x)| < 3 \epsilon$ proving the theorem.
\end{proof}

Let us present a result that allows to calculate the adelic intersection product over a GVF structure that comes from a polarisation, due to Chen and Moriwaki.

Let $(S, \ov{H}_1, \dots, \ov{H}_n)$ be a polarisation inducing a GVF structure on $\Field = \QQ(S)$. Let $X$ be a $d$-dimensional projective variety over $\Field$ which is the generic fiber of a projective morphism $\pi:\cX \to S$. Fix globally integrable line bundles $\ov{\cL}_0, \dots, \ov{\cL}_d \in \intPicQ(\cX/S)_{\QQ}$ and denote by $\ov{L}_0, \dots, \ov{L}_d$ their restriction to $\intPicQ(X)_{\Field}$.
\begin{proposition}~\cite[Proposition 4.5.1]{Adelic_curves_2}~\label{proposition_calculating_generic_intersection}
    The following equality holds:
    \[ \adeg(\ov{L}_0, \ldots, \ov{L}_d|X) = \ov{\cL}_0 \cdot \ldots \cdot \ov{\cL}_d \cdot \pi^*\ov{H}_1 \cdot \ldots \cdot \pi^*\ov{H}_n, \]
    where the left hand side is the adelic intersection product over the globally valued field $\Field$ and the right hand side is the arithmetic/Arakelov intersection product of integrable Zhang divisors on $\cX$.
\end{proposition}
\begin{proof}
    The case of the polarisation and the line bundles $\ov{L}_i$ being defined on a model over $\bbZ$ is \cite[Proposition 4.5.1]{Adelic_curves_2}. Our version follows from continuity of the intersection product on semipositive line bundles and continuity of the intersection number in families, cf.\ Theorem \ref{thm:continuity_of_intersection_number}.
\end{proof}

\subsection{The trivial absolute value}

The discussion so far does not allow for non-trivial norms at trivial absolute values. However, this is something natural to do and does not affect the continuity result. We adapt some notions to the trivially valued case.

Let $K$ be a trivially valued field. Fix a projective morphism $\pi:\cX \to S$ of finite type $K$-schemes, where $S$ is not necessarily projective. When $V$ is a normed vector space over $K$, then $\ov{\cO(1)}$ on $\bbP(V)$ is endowed with a canonical quotient metric. Over a trivially valued field, we give the following variant of Definition~\ref{definition:simple_line_bundles} where the only difference is that we allow an arbitrary norm on $V$.

\begin{definition}
    We say that a metrized line bundle $\ov{\cL}$ on $\cX$ is \textit{simple over $S$}, if there is a closed embedding $j:\cX \to \bbP(V)_S$ over $S$ (for some normed vector space $V$), such that $\ov{\cL} = j^*\ov{\cO(1)}$. The elements of the $\QQ$-vector space of metrized $\QQ$-line bundles on $\cX$ generated by simple ones are called \textit{lattice line bundles on $\cX$ over $S$} and are denoted by $\LPicQ(\cX/S)$. The space of metrized $\QQ$-divisors coming from rational sections of such metrized $\QQ$-line bundles is denoted by $\LdivQ(\cX/S)$ and its elements are called \textit{lattice divisors on $\cX$ over $S$}. If $S$ is equal to $\Spec(K)$, we omit it in the notation. Moreover, if we want to emphasize the dependence on $K$, we use the notation $\LPicQ(\cX/S)_K$.
\end{definition}

We employ approximation to obtain a bigger class of line bundles. The metrized line bundles we obtain are closely related to well-known classes of metrics, cf.\ \cite[Section 2.3]{Adelic_curves_1}. Since $S$ is compact in the constructible topology, we can simplify the definition compared to Definition \ref{definition:global_line_bundles}

\begin{definition}~\label{definition:trivial_global_line_bundles}
    A \emph{global line bundle} $\ov{\sL}$ on $\sX$ over $S$ is defined to be a line bundle $\sL$ with a metric $\phi$ such that there is a sequence of lattice line bundles $\ov{\sL}_k=(\sL,\phi_k)$ such that
    \[
        \lim_{k \to\infty} \|\phi_k-\phi\|_{\sup} = 0.
    \]
    It is called semipositive if the metrics $\phi_k$ can be chosen semipositive on fibres. A global line bundle $\ov{\sL}$ is called integrable if there are semipositive line bundles $\ov{\sL}_+$ and $\ov{\sL}_-$ and an isometry $\ov{\sL} \cong \ov{\sL}_+ \otimes \ov{\sL}^\vee_-$.

    We denote the isometry classes of global $\bbQ$-line bundles by $\aPicQ(\cX/S)$. The subgroup of integrable line bundles is denoted by $\intPicQ(\cX/S)$. If $S = \Spec K$, we often omit it in the notation.
\end{definition}

Proving the following theorem allows us to consider trivial absolute values for general GVFs.

\begin{theorem}
    Let $\cX \to S$  be a flat projective morphism of finite type schemes over a trivially valued field $K$ of relative dimension $d$. Let $\ov{\sL}_0,\dots,\ov{\sL}_d\in \intPicQ(\cX/S)$ be integrable line bundles on $\cX$ over $S$. Then, the map
    \begin{align*}
        S &\to \bbR\\
        s&\mapsto\adeg(\ov{\sL}_0,\dots,\ov{\sL}_d|\sX_s)
    \end{align*}
    is continuous with respect to the constructible topology on $S$.
\end{theorem}

\begin{proof}
We follow the notation and proof strategy from section \ref{sec:definability}. Let $\pi:\cX \to S$ be a flat projective morphism with $d$-dimensional fibers. Suppose $S = \Spec A$ is an affine variety. Let $\ov{\cL}_0, \dots, \ov{\cL}_d \in \LPic(\cX/S)$ be simple over $S$ with embeddings $\alpha_i:\cX \to \PP(V_i)_S$. For $s \in S$ and any object $Q$ over $S$, the notation $Q(s)$ denotes its base change to $s$. Denote by $\delta_i$ the intersection number
\[ \deg(\cL_0(s) \cdot \ldots \cdot \cL_{i-1}(s) \cdot \cL_{i+1}(s) \cdot \ldots \cdot \cL_d(s)|\cX_s)\]
for any $s \in S(\Field)$ in any field extension $K \subset \Field$. It is independent of the choice of $s$ since the family $\pi:\cX \to S$ is flat.

We invoke Lemma \ref{lemma:uniform_resultants} to obtain a polynomial $R=R_1\in A \otimes_K S^{\delta_0}(V_0^{\vee}) \otimes_K \ldots \otimes_K S^{\delta_d}(V_d^{\vee})$ such that $R(s)$ is the resultant for all $s\in S$. By applying \cite[Remark 3.9.3]{Adelic_curves_2}, we compute the intersection number to be 
\[
    \adeg(\ov{\sL}_0,\dots,\ov{\sL}_d|\sX_s) = \log \|R(s)\|.
\]
Since $K$ is trivially valued, the norm on $S^{\delta_0}(V_0^{\vee}) \otimes_K \ldots \otimes_K S^{\delta_d}(V_d^{\vee})$ has an orthogonal basis. It follows that $\log \|R(\blank)\|$ is continuous with respect to the constructible topology. In fact, it is even lower semicontinuous with respect to the Zariski topology.

The approximation statement in Theorem \ref{thm:continuity_approximation} applies to prove the statement. 
\end{proof}

\begin{remark}
    For the height with respect to semipositive line bundles $\sL_0,\dots,\sL_d$ the height is lower semicontinuous with respect to the Zariski topology. This follows immediately from the proof.
\end{remark}

\section{Average intersections}~\label{sec:Average_intersections}
Let $f_1, \dots, f_m$ be Laurent polynomials in $n$ variables with coefficients in a number field $K$. Each $f_i$ defines a hypersurface $V_i$ inside a proper toric variety $T$ with torus $\bbT=\bbG_m^n \subset T$. Let $(u_{1,j},\dots,u_{m,j})_j$ be a generic sequence of small points in $\bbT^m$ with respect to the Weil height on $\bbT^m \subset \PP^{nm}$. For integrable Zhang divisors $\ov{D}_0,\dots,\ov{D}_{n-m}$ on $T$ we want to compute
\[
\lim_{j\to\infty} \adeg(\ov{D}_0,\dots,\ov{D}_{n-m}|u_{1,j}V_1\cap\dots\cap u_{m,j}V_m).
\]
Denote the coordinates of the $i$-th factor of $\bbG_m^n$ by $w_{1,i},\dots,w_{n,i}$. We let $V \subset T \times \bbT^m$ be the intersection of the vanishing loci of $f_i(z_1w^{-1}_{1,i},\dots, z_n w^{-1}_{n,i})$. We note that under $V \to \bbT^m$ the generic fibre has dimension $n-m$. The map $V \to \bbT^m$ is flat of relative dimension $n-m$ over a dense Zariski open $U \subseteq \bbT^m$. We define global line bundles $\ov{\sL}_0,\dots,\ov{\sL}_{n-m}\in \intPicQ(V/U)$ by pulling back $\sO(\ov{D}_0),\dots,\sO(\ov{D}_{n-m})$ to $T\times \bbT^m$ and restricting to $V$. This makes sense by Theorem~\ref{theorem:Zhang_divisors_are_induced_by_global_divisors}. By Theorem \ref{thm:continuity_of_intersection_number}, the map
\begin{align*}
        U_{\GVF} &\to \bbR\\
        u&\mapsto\adeg(\ov{\sL}_0,\dots,\ov{\sL}_d|\sX_u)
\end{align*}
is continuous. We note that on $(u_1,\dots,u_m) \in U(\Kbar)$ the above map is given by
\[
    (u_1,\dots,u_m) \mapsto \adeg(\ov{D}_0,\dots,\ov{D}_{n-m}|u_1 V_1\cap\dots\cap u_m V_m).
\]

By Corollary~\ref{corollary:Yuan_equidistribution_as_GVF_convergence_for_projective_space}, a generic small sequence $(u_{1,j},\dots,u_{m,j})_j$ in $\bbT^m$ has the corresponding points on $U_{\GVF}$ converging to $K(w_{1,1},\dots,w_{n,1},\dots,w_{1,m},\dots,w_{n,m})$ with the polarized GVF structure associated to $(\prod_{i=1}^m \bbP^{n},\pi_1^* \ov{\sO(1)}^n,\dots,\pi_m^* \ov{\sO(1)}^n)$. Denote this limit point by $\eta^{\can} \in U_{\GVF}$.

Since the right hand side is the intersection over a polarized GVF it can be computed up to birational modification as the height of $V \subset T \times \prod^m_{i=1} \bbP^n$ with respect to $\pi_h^*\ov{D}_0\dots\pi_h^*\ov{D}_{n-m}\cdot(\sum\pi_i^* \ov{\sO(1)})^{nm}$ by Proposition \ref{proposition_calculating_generic_intersection}. In other words, we get the following.

\begin{lemma}~\label{lemma:the_limit_of_intersections_as_the_polarised_intersection}
    Suppose that $\ov{D}_0,\dots,\ov{D}_{n-m}$ are integrable Zhang divisors on $T$ over $K$. Then,
    \[
        \lim_{j\to\infty} \adeg(\ov{D}_0,\dots,\ov{D}_{n-m}|u_{1,j}V_1\cap\dots\cap u_{m,j}V_m) 
    \]
    \[  
        = \adeg(\pi_h^*\ov{D}_0\dots\pi_h^*\ov{D}_{n-m}\cdot\pi_1^* \ov{\sO(1)}^{n}\cdots\pi_m^* \ov{\sO(1)}^{n} |V).
    \]
    On the right hand side, the $\ov{D}_i$ should be viewed as adelic divisors pulled back to $F$.
\end{lemma}
We need to be slightly careful when applying Fubini's theorem in the non-Archimedean setting. This is because of the failure of $(X \times Y)^{\an} = X^{\an} \times Y^{\an}$. In our setting, there is a preferred very affine chart given by the torus in the toric variety on which we may apply Fubini.

We sketch an argument that one can always apply Fubini in such situations. This is based on \cite[Proposition 3.4.21]{stoffel_thesis}. If $\alpha \in A^{(\dim X, \dim X)}(X^{\an})$ and $\beta \in A^{(\dim Y, \dim Y)}(Y^{\an})$ are smooth forms that are defined on very affine charts of integration $U\subseteq X$ and $V \subseteq Y$. Then, $U\times V \subseteq X\times Y$ is a very affine chart of integration for $\pi_X^*\alpha\wedge \pi_Y^*\beta$. Furthermore, $\trop(U\times V) = \trop(U) \times \ \trop(V)$ by Rosenlicht's theorem. Then, one needs to prove that one has the product measure on each polyhedron in the tropicalization. This is done by adapting \cite[Lemma 3.4.16]{stoffel_thesis}. The general case follows by approximation. We refer to \cite{Gubler_forms_and_currents} for an introduction to the theory of forms in the non-Archimedean setting. 
\begin{theorem}
\label{thm:ronkin_integral_vanishing}
    Let $\sR_i$ denote the Ronkin divisor associated to $g_i = f_i(z_1w^{-1}_{1,i},\dots, z_n w^{-1}_{n,i})$ on a suitable toric blowup $X$ of $T \times \prod^m_{i=1} \bbP^n$. Let $\tilde{V}$ denote the common vanishing locus of the $g_i$. Then, we have an identity
    \[
    \adeg(\sR_1\dots \sR_m \pi_h^*\ov{D}_0\dots\pi_h^*\ov{D}_{n-m}\cdot\pi_1^* \ov{\sO(1)}^{n}\cdots\pi_m^* \ov{\sO(1)}^{n}|X) 
    \]
    \[
    = \adeg(\pi_h^*\ov{D}_0\dots\pi_h^*\ov{D}_{n-m}\cdot\pi_1^* \ov{\sO(1)}^{n}\cdots\pi_m^* \ov{\sO(1)}^{n}|\tilde{V}).
    \]
\end{theorem}

\begin{proof}
    In order not to overburden notation we omit the superscript denoting the analytification.
    
    Let $s_i$ denote the distinguished section of $\sO(\sR_i)$. The sections $g_i s_i$ of $\sO(\sR_i)$ have common vanishing locus $\tilde{V}$. By the iterative definition of the height, the equality of intersection numbers is equivalent to the vanishing of the integrals occurring in the height computations. These are of the form 
    \[
    \int_{\cdiv(g_1)\cap\dots\cap\cdiv(g_{r-1})} \log|g_r s_r| c_1(\sR_{r+1})\dots c_1(\sR_m)\pi_h^* c_1(\ov{D}_0)\dots \pi_h^* c_1(\ov{D}_{n-m})\prod^m_{i=1}\pi_i^*c_1(\ov{\sO(1)}).
    \]
    Let us write $\omega$ for $c_1(\sR_{r+1})\dots c_1(\sR_m)\pi_h^* c_1(\ov{D}_0)\dots \pi_h^* c_1(\ov{D}_{n-m})\prod_{i\neq r}\pi_i^*c_1(\ov{\sO(1)})$.
    To each of them we can apply Fubini to obtain
    \begin{align*}
       &\int_{\cdiv(g_1)\cap\dots\cap\cdiv(g_{r-1})\cap X} \log|g_r s_r| \pi_r^*c_1(\ov{\sO(1)})^n\omega\\
       =&\int_{\left(\cdiv(g_1)\cap\dots\cap\cdiv(g_{r-1})\cap(\bbT\times \prod^{r-1}_{i=1} \bbT)\right)\times \bbT\times\prod^m_{i=r+1} \bbT} \log|g_r s_r| \pi_r^*c_1(\ov{\sO(1)})^n\omega\\
       =&\int_{\left(\cdiv(g_1)\cap\dots\cap\cdiv(g_{r-1})\cap(\bbT\times \prod^{r-1}_{i=1} \bbT)\right)\times \prod^m_{i=r+1} \bbT} \left(\int_{\bbT}\log|g_r s_r| c_1(\ov{\sO(1)})^n\right)\omega\\
       =&\int_{\left(\cdiv(g_1)\cap\dots\cap\cdiv(g_{r-1})\cap(\bbT\times \prod^{r-1}_{i=1} \bbT)\right)\times \prod^m_{i=r+1} \bbT} \left(\int_{\bbT}\log|g_r s_r| d\sigma_0(x)\right)\omega.
    \end{align*}
    The first equality follows since a Zariski closed subset with empty interior is a nullset with respect to a measure associated to differential forms.

    We claim that the inner integral $\int_{\bbT}\log|g_r s_r| \sigma_0(x)$ vanishes at each fiber. Recall that $g_r(t,t_1,\dots,t_m) = f_r(t t^{-1}_r)$. Let $\pi_r$ denote the projection from $\bbT \times \prod^{m}_{i=1} \bbT$ to the $r$-th component $\bbT_r$ in the second factor. 
    
    We compute
    \begin{align*}
        \log |s_r(t,t_1,\dots,t_m)| =& \int_{\trop^{-1}(\trop(t,t_1,\dots,t_m))} -\log|g_r(x)| d \sigma_{(t,t_1,\dots,t_m)}(x)\\ =& \int_{\trop^{-1}(\trop(t,t_r))\subset (\bbT\times\bbT_r)^{\trop}} -\log| f_r(xx_r^{-1})| d \sigma_{(t,t_r)}(x,x_r)\\ =& \int_{\trop^{-1}(\trop(tt_r^{-1}))\subset \bbT^{\trop}_r} -\log| f_r(x)| d \sigma_{tt_r^{-1}}(x) = \rho_r(tt_r^{-1}).
    \end{align*}
    Here, $\rho_r$ denotes the Ronkin function for $f_r$ and $\bbT_r$ denotes the $r$-th factor of the torus.
    
    Consider the fibre over an element 
    \[
        (t,t_1,\dots,t_{r-1},t_{r+1},\dots,t_m) \in \left(\cdiv(g_1)\cap\dots\cap\cdiv(g_{r-1})\cap(\bbT\times \prod^{r-1}_{i=1} \bbT)\right)\times \prod^m_{i=r+1} \bbT.
    \]
    Over this fibre we evaluate the integral
    \begin{align*}
        \int_{\bbT}\log|g_r(t,t_1,\dots,t_m) s_r(t,t_1,\dots,t_m)| \sigma_0(t_r) &= \int_{\bbT} \log|f_r(tt_r^{-1})|+\rho_r(tt_r^{-1}) \sigma_0(t_r)\\ &=-\rho_r(t) + \rho_r(t) = 0.
    \end{align*}
\end{proof}

\begin{lemma}
\label{lemm:intersection_pushforward}
    Let $R_i$ denote the Ronkin line bundle associated to $f_i$ and suppose it is already defined on $T$ and assume $\ov{D}_0,\dots,\ov{D}_{n-m}$ are toric. Let $X$ denote a suitable toric blow-up of $T\times\prod^m_{i=1} \bbP^n$ over which $\sR_i$ are defined. We have an equality of intersection numbers
    \[
    \adeg(\sR_1\dots \sR_m \pi_h^*\ov{D}_0\dots\pi_h^*\ov{D}_{n-m}\cdot\pi_1^* \ov{\sO(1)}^{n}\cdots\pi_m^* \ov{\sO(1)}^{n}|X) 
    \]
    \[
    = \adeg(R_1\dots R_m \ov{D}_0\dots\ov{D}_{n-m}|T)
    \]
\end{lemma}

\begin{proof}
    By linearity, we assume that $\ov{D}_0,\dots,\ov{D}_{n-m}$ are all semipositive. Then, we interpret the left hand side in a combinatorial manner as in \cite[Theorem 5.2.5]{burgos_philippon_sombra_toric_varieties}. We then apply Lemma \ref{lemm:convex_pushforward}. The occurring polytopes are $m$ times $n$ copies of the unit simplex, one set of copies for each factor in $\prod^m_{i=1} \bbP^n$. It is immediate to see that the pushforward of the roof functions of the $\pi_h^*\ov{D}_i$ yield precisely the roof functions of the $\ov{D}_i$. Similarly, the pushforward of the roof function of $\sR_i$ is the roof function of $R_i$ by Lemma \ref{lemm:pushforward_of_ronkin}.
\end{proof}

\begin{theorem}
    \label{thm:intersection}
    Let $f_1, \dots, f_m$ be Laurent polynomials in $n$ variables with coefficients in a number field $K$ and let $T$ be a proper toric variety with torus $\bbT=\bbG_m^n \subset T$. Suppose that $NP(f_i)$ define divisors on $T$. Denote by $V_i$ the hypersurface defined by $f_i$. Let $(\zeta_{1,j},\dots,\zeta_{m,j})_j$ be a generic sequence of small points in $\bbT^m$ with respect to the Weil height and $\ov{D}_0,\dots,\ov{D}_{n-m}$ be integrable toric Zhang divisors on $T$. Then,
    \[
        \lim_{j\to\infty} \adeg(\ov{D}_0,\dots,\ov{D}_{n-m}\mid \zeta_{1,j}V_1\cap\dots\cap\zeta_{m,j}V_m)=
        \adeg(R_1\dots R_m \ov{D}_0\dots\ov{D}_{n-m}|T).
    \]
\end{theorem}

\begin{proof}
    This is a combination of Lemma \ref{lemma:the_limit_of_intersections_as_the_polarised_intersection}, Theorem \ref{thm:ronkin_integral_vanishing} and Lemma \ref{lemm:intersection_pushforward}. 
\end{proof}

We finally prove Conjecture 6.4.4. in \cite{gualdi:tel-01931089}.

\begin{theorem}\label{thm:gualdi_conjecture}
    Let $f_1, \dots, f_m$ be Laurent polynomials in $n$ variables with coefficients in a number field $K$ and let $T$ be a proper toric variety with torus $\bbT=\bbG^n \subset T$. Denote by $V_i$ the hypersurface defined by $f_i$ and by $\rho_i$ its Ronkin function. Let $(\zeta_{1,j},\dots,\zeta_{m,j})_j$ be a generic sequence of small points in $\bbT^m$ for the Weil height and let $\ov{D}_0,\dots,\ov{D}_{n-m}$ be semipositive toric Zhang divisors on $T$  with associated local roof functions $\theta_{0,v}, \dots, \theta_{n-m,v}$. Then,
    \[
        \lim_{j\to\infty} \adeg(\ov{D}_0, \dots,\ov{D}_{n-m}\mid \zeta_{1,j}V_1\cap\dots\cap\zeta_{m,j}V_m)
    \]
    \[
        =\sum_{v\in \M_K} n_v \MI_M(\theta_{0,v}, \dots, \theta_{n-m,v}, \rho_1^\vee, \dots, \rho_m^\vee).
    \]
\end{theorem}

\begin{proof}
        We apply the projection formula to restrict to the case, where the $NP(f_i)$ define divisors on $T$. Then, the conjecture follows from Theorem \ref{thm:intersection} and Theorem~\ref{thm:convex_formula_for_toric_intersection}.
\end{proof}

\appendix

\section{Mahler measures of complex polynomials}~\label{appendix:A}

In this appendix, we study some measures of complexity of complex polynomials. For a nonzero $P\in \mathbb{C}[X_1,\ldots,X_n]$, we define the logarithmic Mahler measure
\[
    m(P) = \int_{[0,1]^n} \log|P(e^{2i\pi t_1},\ldots,e^{2i\pi t_n})| dt_1\ldots dt_n,
\]
and the logarithmic Fubini-Study Mahler measure
\[
    m_{\mathbb{S}_n}(P) = \int_{\mathbb{S}_n} \log|P(z_1,\ldots,z_n)| d\eta_n(z_1,\ldots,z_n),
\]
where $\mathbb{S}_n$ is the unit sphere in $\mathbb{C}^n$ for the usual Euclidean norm, and $\eta_n$ is the spherical measure on $\mathbb{S}_n$, normalized so that $\eta_n(\mathbb{S}_n) = 1$.

In \cite{Lelong_Constantes_Universelles}, Pierre Lelong studied these two measures and gave a bound for the distance between them in terms of $n$ and the degree of $P$.

In this appendix, we prove an analogue of Lelong’s result for polynomials in $\overline{X}_1,\ldots,\overline{X}_n$, where each $\overline{X}_i$ is a tuple of abstract variables of length $m_i$. More precisely, we define the mixed Fubini-Study Mahler measure by

\[
    m_{\mathbb{S}_{m_1}\times\ldots\times\mathbb{S}_{m_n}}(P) = \int_{\mathbb{S}_{m_1}\times\ldots\mathbb{S}_{m_n}} \log|P(\overline{z}_1,\ldots,\overline{z}_n)| d\eta_{m_1}(\overline{z}_1)\wedge \ldots\wedge d\eta_{m_n}(\overline{z}_n).
\]

This measure has already been studied by Rémond in \cite{remond_multiprojective_diophantine_geometry}, in which he bounds $m_{\mathbb{S}_{m_1}\times\ldots\times\mathbb{S}_{m_n}}(P)$ in terms of the absolute values of the coefficients of $P$.
Our goal here is to prove a slightly different inequality, namely the following proposition.

\begin{proposition}
\label{proposition:comparison_fsmahler_norm_multi}
Let $P \in \mathbb{C}[\overline{X}_1,\ldots,\overline{X}_n]$ be a nonzero polynomial, where each $\overline{X}_i$ is a tuple of abstract variables of length $m_i$. For all $i\leq n$, let $d_i$ be the degree of $P$ in $\overline{X}_i$. Then,
\[
    \left|m_{\mathbb{S}_{m_1}\times\ldots\times \mathbb{S}_{m_n}} - \log\lVert P\rVert\right| \leqslant \sum_{i=1}^n d_i\left(\log(m_i+1)+\frac{1}{2}\sum_{k=1}^{m_i-1}\frac{1}{k}\right),
\]
where $\lVert P \rVert$ is the maximum absolute value of the coefficients of $P$.
\end{proposition}

Let us start with the simpler case where each $m_i$ is equal to $1$. Let $P\in\mathbb{C}[X_1,\ldots,X_n]$ be a nonzero polynomial of degree $d$.

\begin{lemma}
\label{lemma:norm_smoothing}
Let
\[
    S(P) := \sup\{|P(z_1,\ldots,z_n)| \,:\,  (z_1,\ldots,z_n)\in\mathbb{C}^n \text{ and } |z_i|\leqslant 1 \text{ for all } i\}.
\]
Then,
\[
    \lVert P \rVert \leqslant S(P) \leqslant \binom{n+d}{n}\lVert P\rVert.
\]
\end{lemma}

\begin{proof}
The rightmost inequality follows from the fact that a polynomial of degree $d$ has at most $\binom{n+d}{n}$ nonzero coefficients.

For the other inequality, consider the integral
\begin{equation*}
\begin{aligned}
    I &= \int_{[0,1]^n} |P(e^{2i\pi t_1},\ldots,e^{2i\pi t_n})|^2\, dt_1\ldots dt_n \\
    &= \int_{[0,1]^n} P(e^{2i\pi t_1},\ldots,e^{2i\pi t_n})\overline{P(e^{2i\pi t_1},\ldots,e^{2i\pi t_n})}\, dt_1\ldots dt_n \\
    &= \int_{[0,1]^n} \left(\sum_{k,l\in\mathbb{N}^n}a_k\overline{a_l}\exp\left(2i\pi(k_1-l_1)t_1+\ldots+2i\pi(k_n-l_n)t_n\right)\right)\, dt_1\ldots dt_n \\
    I &= \sum_{k,l\in\mathbb{N}^n}a_k\overline{a_l}\int_{[0,1]^n}\exp\left(2i\pi(k_1-l_1)t_1+\ldots+2i\pi(k_n-l_n)t_n\right)\, dt_1\ldots dt_n.
\end{aligned}
\end{equation*}

Now, for $k,l\in\mathbb{N}^n$ with $k\neq l$, there exists $1\leqslant r \leqslant n$ with $k_r-l_r\neq 0$. So,
\begin{multline*}
    \int_{[0,1]^n}\exp\left(2i\pi(k_1-l_1)t_1+\ldots+2i\pi(k_n-l_n)t_n\right)\, dt_1\ldots dt_n 
    \\ = \int_{[0,1]^{n-1}} \exp\left(\sum_{s\neq r}2i\pi(k_s-l_s)t_s\right)\left(\int_0^1 e^{2i\pi(k_r-l_r)t_r} dt_r\right)\prod_{s\neq r}dt_s 
    = 0,
\end{multline*}
and for $k=l$,
\[
    \int_{[0,1]^n}\exp\left(2i\pi(k_1-l_1)t_1+\ldots+2i\pi(k_n-l_n)t_n\right)\, dt_1\ldots dt_n = 1.
\]

So, finally
\[
    I = \sum_{k\in\mathbb{N}} |a_k|^2 \geqslant \lVert P\rVert^2.
\]
But we also have
\[
    I = \int_{[0,1]^n} |P(e^{2i\pi t_1},\ldots,e^{2i\pi t_n})|^2\, dt_1\ldots dt_n \leqslant S(P)^2.
\]
Hence, $\lVert P \rVert \leqslant S(P)$.
\end{proof}

\begin{remark}
In particular, Lemma \ref{lemma:norm_smoothing} implies that
\[
    S(P) = \lim_{m\rightarrow +\infty} \lVert P^m\rVert^{1/m}.
\]
\end{remark}

\begin{lemma}
\label{lemma:comparison_mahler_coeffs_1}
Assume that $n=1$, i.e. $P = \sum\limits_{k=0}^{d} a_k X^k \in\mathbb{C}[X]$. Then, for all $0\leqslant k\leqslant d$:
\[
    |a_k| \leqslant \binom{d}{k}\exp(m(P))
\]
\end{lemma}

\begin{proof}
Since $\mathbb{C}$ is algebraically closed, we can write $P = \lambda \prod\limits_{i=1}^{\deg P} (X-\alpha_i)$, where $\lambda\in \mathbb{C}^{\times}$, $\alpha_1,\ldots,\alpha_{\deg P}\in \mathbb{C}$. Now, by Jensen's formula, we have for each $1\leqslant i\leqslant \deg P$,
\[
    \int_0^1 \log\left|e^{2i\pi t}-\alpha_i\right|dt = \max(0,\log|\alpha_i|).
\]
So, by summing,
\[
    m(P) = \log|\lambda| + \sum_{i=1}^{\deg{P}} \max(0,\log|\alpha_i|).
\]
Now, let $k\leqslant \deg P$. Then, the coefficient of $X^k$ in $P$ is equal to
\[
    a_k = (-1)^{\deg P - k}\lambda \sum_{I\subseteq \{1,\ldots,\deg P\}} \prod_{i\in I}\alpha_i.
\]
So, by the triangle inequality
\[
    |a_k| \leqslant |\lambda|\sum_{I\subseteq \{1,\ldots,\deg P\}}\prod_{i\in I}|\alpha_i| \leqslant  \binom{\deg P}{k}|\lambda|\prod_{i=1}^{\deg P}\max(1,\alpha_i) \leqslant \binom{d}{k}\exp(m(P).
\]
\end{proof}

\begin{lemma}
\label{lemma:comparison_mahler_coeffs}
Write $P = \sum\limits_{m\in\mathbb{N}^n} a_m X_1^{m_1}\ldots X_n^{m_n}$. Then, for every $m = (m_1,\ldots,m_n)\in\mathbb{N}^n$, we have
\[
    |a_m| \leqslant \binom{d}{m_1,\ldots,m_n}\exp(m(P)).
\]
\end{lemma}

\begin{proof}
We prove this inequality by induction on $n$. The $n=1$ case is the result of Lemma \ref{lemma:comparison_mahler_coeffs_1}, so we may assume $n\geqslant 2$. The result is also immediate if $a_m$ is zero, so assume it is not.
Write $P = \sum\limits_{k=0}^{d} P_k X_n^k$, where the $P_k$ are in $\mathbb{C}[X_1,\ldots,X_{n-1}]$.
Then, we may write
\[
    m(P) = \int_{[0,1]^{n-1}} m(P(e^{2i\pi t_1},\ldots,e^{2i\pi t_{n-1}},X)) dt_1\ldots dt_{n-1}.
\]
By Lemma \ref{lemma:comparison_mahler_coeffs_1}, we have for all $(t_1,\ldots,t_{n-1})\in [0,1]^{n-1}$, $|P_{m_n}(e^{2i\pi t_1},\ldots,e^{2i\pi t_{n-1}})| \leqslant \binom{d}{m_n} \exp(m(P(e^{2i\pi t_1},\ldots,e^{2i\pi t_{n-1}},X)))$. Since $P_{m_n}(e^{2i\pi t_1},\ldots,e^{2i\pi t_{n-1}})$ is nonzero for almost all $(t_1,\ldots,t_{n-1})$, we may write\\$m(P(e^{2i\pi t_1},\ldots,e^{2i\pi t_{n-1}},X)) \geqslant \log|P_{m_n}(e^{2i\pi t_1},\ldots,e^{2i\pi t_{n-1}})| - \log\binom{d}{m_n}$ and integrate, yielding
\[
    m(P) \geqslant m(P_{m_n}) - \log\binom{d}{m_n},
\]
i.e., $\exp(m(P_{m_n})) \leqslant \binom{d}{m_n} \exp(m(P))$
Since $P_{m_n}$ has degree at most $d-m_n$, the induction hypothesis gives $|a_m| \leqslant\binom{d-m_n}{m_1,\ldots,m_{n-1}}\exp(m(P_{m_n}))$. Since $\binom{d}{m_n}\binom{d-m_n}{m_1,\ldots,m_{n-1}} = \binom{d}{m_1,\ldots,m_n}$, this concludes.
\end{proof}

\begin{corollary}
\label{corollary:comparison_mahler_norm}
\[
    |m(P) - \log\lVert P\rVert| \leqslant d\log(n+1)
\]
\end{corollary}

\begin{proof}
First, it is clear that $m(P) \leqslant \log(S(P)) \leqslant \log\lVert P\rVert + \log\binom{n+d}{n}$ by Lemma \ref{lemma:norm_smoothing}. A basic counting argument shows that $\binom{n+d}{n}\leqslant (n+1)^d$, hence $m(P) \leqslant \log\lVert P\rVert + d\log(n+1)$.

For the other direction, write $P = \sum\limits_{m\in\mathbb{N}^n} a_m X_1^{m_1}\ldots X_n^{m_n}$. By Lemma \ref{lemma:comparison_mahler_coeffs}, we have for all $m \in\mathbb{N}^n$,
\[
    |a_m| \leqslant \binom{d}{m_1,\ldots,m_n}\exp(m(P))
\]
A basic counting argument shows that $\binom{d}{m_1,\ldots,m_n}\leqslant (n+1)^d$, hence by taking the maximum over all coefficients,
\[
    \lVert P\rVert \leqslant (n+1)^d\exp(m(P)),
\]
i.e. $\log\lVert P\rVert \leqslant m(P) + d\log(n+1)$, which concludes the proof.
\end{proof}

\begin{remark}
If $P\in \mathbb{C}[X_1,\ldots,X_n]$ is homogeneous, we may replace $n+1$ by $n$ in the above inequality. Indeed, evaluating in $X_n = 1$ does not change $m(P)$ and $\lVert P\rVert$, so we may replace $P$ with a polynomial in $n-1$ variables.
\end{remark}

\begin{lemma}
\label{lemma:comparison_fsmahler_norm}
\[
    \left|m_{\mathbb{S}_n}(P) - \log\lVert P\rVert\right| \leqslant 2d\log(n+1).
\]
\end{lemma}

\begin{proof}
It is clear from the definition that $m_{\mathbb{S}_n}(P(e^{2i\pi t_1}X_1,\ldots,e^{2i\pi t_n}X_n)) = m_{\mathbb{S}_n}(P)$ for all $t_1,\ldots,t_n\in [0,1]$. So, we may write
\[
    m_{\mathbb{S}_n}(P) = \int_{[0,1]^n} m_{\mathbb{S}_n}(P(e^{2i\pi t_1}X_1,\ldots,e^{2i\pi t_n}X_n)) dt_1\ldots dt_n.
\]
By Fubini, this is equal to
\[
    \int_{\mathbb{S}_n} \left(\int_{[0,1]^n} (P(z_1e^{2i\pi t_1},\ldots,z_ne^{2i\pi t_n})) dt_1\ldots dt_n\right) d\eta_n(\overline{z}) = \int_{\mathbb{S}_n} m(P(z_1X_1,\ldots,z_nX_n)) d\eta_n(\overline{z})
\]
But, by Lemma \ref{corollary:comparison_mahler_norm}, we have for all $\overline{z}\in\mathbb{S}_n$:
\[
    \left|m(P(z_1X_1,\ldots,z_nX_n)) - \log\lVert P\rVert\right| \leqslant d\log(n+1)
\]
So, by integrating:
\[
    \left| m_{\mathbb{S}_n}(P) - \int_{\mathbb{S}_n} \log\lVert P(z_1X_1,\ldots,z_nX_n)\rVert d\eta_n(\overline{z})\right| \leqslant d\log(n+1)
\]

Moreover, it is clear that for all $\overline{z}\in\mathbb{S}_n$, $\lVert P(z_1X_1,\ldots,z_nX_n)\rVert \leqslant \lVert P \rVert$, therefore $\int\limits_{\mathbb{S}_n} \log\lVert P(z_1X_1,\ldots,z_nX_n)\rVert d\eta_n(\overline{z}) \leqslant \log\lVert P\rVert$.

On the other hand, if $P$ is written as $\sum_{s\in\mathbb{N}^n} a_s X_1^{s_1}\ldots X_n^{s_n}$, we have for all $s$ such that $a_s\neq 0$,
\begin{equation*}
    \begin{aligned}
    \int_{\mathbb{S}_n} \log\lVert P(z_1X_1,\ldots,z_nX_n)\rVert d\eta_n(\overline{z})
    & = \int_{\mathbb{S}_n} \left(\max_s \log|a_sz_1^{s_1}\ldots z_n^{s_n}|\right) d\eta_n(\overline{z})\\
    & \geqslant \max\limits_s \int_{\mathbb{S}_n} \log|a_sz_1^{s_1}\ldots z_n^{s_n}| d\eta_n(\overline{z})\\
    & = \max\limits_s \left(\log|a_s| + \sum\limits_{i=1}^n s_i \int_{\mathbb{S}_n} \log|z_i| d\eta_n(\overline{z}) \right)\\
    & = \max_s \left(\log|a_s| + |s|\int_{\mathbb{S}_n} \log|z_1| d\eta_n(\overline{z})\right)\\
    \int_{\mathbb{S}_n} \log\lVert P(z_1X_1,\ldots,z_nX_n)\rVert d\eta_n(\overline{z}) & \geqslant \log\lVert P\rVert + d\int_{\mathbb{S}_n} \log|z_1| d\eta_n(\overline{z}).
    \end{aligned}
\end{equation*}
It remains to show that $\int\limits_{\mathbb{S}_n} \log|z_1| d\eta_n(\overline{z}) \geq - \log(n+1)$. In fact, we even know from \cite[Equation 2.28]{Lelong_Constantes_Universelles} that this integral evaluates to $-\frac{1}{2}\sum\limits_{i=1}^{n-1}\frac{1}{k}$.
\end{proof}

Now, we move on to the mixed case. Fix a nonzero polynomial $P\in \mathbb{C}[\overline{X}_1,\ldots,\overline{X}_n]$ and denote $d_i := \deg_{\overline{X}_i}P$. 
Our goal is to adapt the result of Lemma \ref{lemma:comparison_fsmahler_norm} and find a bound for the distance between this measure and $\log\lVert P\rVert$ in terms of the $d_i$.

\begin{lemma}
\label{lemma:norms_smoothing_multi}
Again, let
\[
    S(P) := \sup\{|P(\overline{z_1},\ldots,\overline{z}_n)| (\overline{z}_1,\ldots,\overline{z}_n)\in\mathbb{C}^{m_1+\ldots+m_n} \text{ and } |z_{i,j}| \leqslant 1 \text{ for all } i,j\}.
\]
Then,
\[
    \lVert P\rVert \leqslant S(P) \leqslant \left(\prod_{i=1}^n \binom{m_i+d_i}{m_i}\right)\lVert P\rVert
\]
\end{lemma}
\begin{proof}
This follows from the fact that the number of nonzero coefficients of $P$ is at most $\prod\limits_{i=1}^n \binom{m_i+d_i}{m_i}$, by the same argument as in the proof of Lemma \ref{lemma:norm_smoothing}.
\end{proof}

\begin{lemma}
\label{lemma:comparison_mahler_coeffs_multi}
Write
\[
P = \sum_{(k_1,\ldots,k_n)\in\mathbb{N}^{m_1}\times\ldots\times\mathbb{N}^{m_n}} a_{k_1,\ldots,k_n} \overline{X}_1^{k_1}\ldots \overline{X}_n^{k_n},
\]
where $\overline{X}_i^{k_i} := \prod\limits_{j=1}^{m_i} X_{i,j}^{k_{i,j}}$. Then, for every $k = (k_1,\ldots,k_n)\in\mathbb{N}^{m_1}\times\ldots\times \mathbb{N}^{m_n}$, we have
\[
    |a_{k_1,\ldots,k_n}| \leqslant \left(\prod_{i=1}^n \binom{d_i}{k_{i,1},\ldots,k_{i,m_i}}\right)\exp(m(P)).
\]
\end{lemma}
\begin{proof}
The proof is a straightforward induction on $n$, based on Lemma \ref{lemma:comparison_mahler_coeffs}.
\end{proof}

\begin{corollary}
\label{corollary:comparison_mahler_norm_multi}
\[
    |m(P) - \log\lVert P \rVert |\leqslant \sum_{i=1}^n d_i\log(m_i + 1)
\]
\end{corollary}
\begin{proof}
First, it is clear that $m(P) \leqslant \log(S(P)) \leqslant \log\lVert P\rVert + \sum_{i=1}^n\log\binom{m_i+d_i}{m_i}$ by Lemma \ref{lemma:norm_smoothing}. As in the proof of Corollary \ref{corollary:comparison_mahler_norm}, a counting argument shows that $\binom{m_i+d_i}{m_i}\leqslant (m_i+1)^{d_i}$, so $m(P)\leqslant \log\lVert P\rVert + \sum_{i=1}^n d_i\log(m_i+1)$.

Then, it follows from Lemma \ref{lemma:comparison_mahler_coeffs_multi} and the fact that every multinomial coefficient $\binom{d_i}{k_1,\ldots,k_{m_i}}$ is smaller or equal to $(m_i+1)^{d_i}$, that
\[
    \lVert P\rVert \leqslant \left(\prod_{i=1}^n(m_i+1)^{d_i}\right)\exp(m(P)),
\]
i.e. $\log\lVert P\rVert \leqslant m(P) + \sum_{i=1}^n d_i\log(m_i+1)$, which concludes the proof.
\end{proof}

\begin{remark}
If $P\in \mathbb{C}[\overline{X}_1,\ldots,\overline{X}_n]$ is homogeneous in each of the tuples $\overline{X}_i$, we may replace $m_i+1$ by $m_i$ in the above inequality. Indeed, evaluating in $X_{i,m_i} = 1$ does not change $m(P)$ nor $\lVert P\rVert$, so we may replace $\overline{X}_i$ by a $(m_i-1)$-tuple.
\end{remark}

We are finally able to prove the main result of this appendix.

\begin{proof}[Proof of Proposition \ref{proposition:comparison_fsmahler_norm_multi}]
We first prove the inequality
\[
    m_{\mathbb{S}_{m_1}\times\ldots\times \mathbb{S}_{m_n}}(P) \leqslant \log\lVert P\rVert + \sum_{i=1}^n d_i\log(m_i+1)
\]
exactly as in the proof of Corollary \ref{corollary:comparison_mahler_norm_multi}.

 For the other inequality, we again reason by induction on $n$. If $n=1$, the result follows directly from Lemma \ref{lemma:comparison_fsmahler_norm}. So, assume $n\geqslant 2$. Write $P = \sum\limits_{k\in\mathbb{N}^{m_n}} P_k \overline{X}_n^k$, where the $P_k$ are in $\mathbb{C}[\overline{X}_1,\ldots,\overline{X}_{n-1}]$. Let $l\in\mathbb{N}^{m_n}$ be such that $\lVert \mathbb{P}_l\rVert = \lVert P\rVert$. Then, we have
\[
    m_{\mathbb{S}_{m_1}\times\ldots\times \mathbb{S}_{m_n}} = \int_{\mathbb{S}_1\times\ldots\times\mathbb{S}_{m_{n-1}}} m_{S_{m_n}}(P(\overline{z}_1,\ldots,\overline{z}_{n-1},\overline{X})) d\eta_{m_1}(\overline{z}_1) \wedge \ldots \wedge d\eta_{m_{n-1}}(\overline{z_{n-1}})
\]
By Lemma \ref{lemma:comparison_fsmahler_norm}, we have for all $\overline{z}_1,\ldots,\overline{z}_{n-1}$,
\begin{equation*}
\begin{aligned}
    m_{S_{m_n}}(P(\overline{z}_1,\ldots,\overline{z}_{n-1},\overline{X}))
    & \geqslant \log\lVert P(\overline{z}_1,\ldots,\overline{z}_{n-1},\overline{X})\rVert - d_n\left(\log(m_n+1)+\frac{1}{2}\sum_{k=1}^{m_n-1}\frac{1}{k}\right)\\
    & \geqslant \log|P_l(\overline{z}_1,\ldots,\overline{z}_{n-1})| - d_n\left(\log(m_n+1)+\frac{1}{2}\sum_{k=1}^{m_n-1}\frac{1}{k}\right),
\end{aligned}
\end{equation*}
so by integrating, we get
\[
    m_{\mathbb{S}_{m_1}\times\ldots\times \mathbb{S}_{m_n}} \geqslant m_{\mathbb{S}_{m_n}}(P_l) - d_n\left(\log(m_n+1)+\frac{1}{2}\sum_{k=1}^{m_n-1}\frac{1}{k}\right).
\]
But, by induction hypothesis,
\[
    m_{\mathbb{S}_{m_n}}(P_l) \geqslant \log\lVert P_l\rVert - \sum_{i=1}^{n-1} d_i\left(\log(m_i+1)+\frac{1}{2}\sum_{k=1}^{m_i-1}\frac{1}{k}\right),
\]
where $\lVert P_l\rVert = \lVert P\rVert$ by assumption, which concludes the proof.
\end{proof}

\printbibliography

\end{document}